\documentclass[12pt]{amsart}

\usepackage{amssymb,latexsym,amsthm, amsmath, amsfonts, comment, bbm, fullpage, tikz-cd}
\usepackage[all]{xy}
\usepackage{graphicx, comment, libertine, wrapfig}
\usepackage[utf8]{inputenc}
\usepackage[T2A]{fontenc}

\ifx\pdfoutput\undefined
\usepackage{hyperref}
\else
\usepackage[pdftex,colorlinks=true,linkcolor=blue,urlcolor=blue]{hyperref}
\fi

\theoremstyle{plain}

\theoremstyle{definition}
\newtheorem{definition}{Definition}
\newtheorem{theorem}{Theorem}

\newtheorem{lemma}{Lemma}

\newtheorem{corollary}{Corollary}
\newtheorem{proposition}{Proposition}

\newtheorem{remark}{Remark}

\newtheorem*{ack}{Acknowledgements}

\newcommand{\foot}[1]{\marginpar{\tiny\raggedright #1}}

\newcommand{\az}[1]{\foot{AZ: #1}}

\begin{document}

\title[Statistical properties of Lorentz gases on aperiodic tilings, Part 1]
      {Statistical properties of Lorentz gases on aperiodic tilings, Part 1}
      \author{Rodrigo Trevi\~no}
      \address{University of Maryland, College Park}
      \email{rodrigo@umd.edu}
      \author{Agnieszka Zelerowicz}
      \address{University of Maryland, College Park}
      \email{azelerow@umd.edu}
      \date{\today}
      \begin{abstract}
        We consider the Lorentz gas model of category A (that is, with no corners and of finite horizon) on aperiodic repetitive tilings of $\mathbb{R}^2$ of finite local complexity. We show that the compact factor of the collision map has the K property, from which we derive mixing for pattern-equivariant functions as well as the planar ergodicity of the Lorentz gas flow.
      \end{abstract}
      \maketitle
      \section{Introduction and statement of results}
      The Lorentz gas was originally introduced as a model for the movement of electrons in crystals. Throughout most of the 20th century,
the term crystal referred to a solid with a periodic atomic structure.
Following Shechtman's discovery of quasicrystals, there is now consensus \cite{lifshitz:what} that crystal is any substrate whose atomic structure has long range order, which includes not only materials whose atomic structure is periodic but also quasicrystals. These are materials whose atomic structure is not modeled by any periodic structure but that nonetheless exhibits long range order.

      
      The study of Lorentz gases has a long history and rich theory and has been the motivation for the development of many tools in ergodic theory and mathematical physics. The list of surveys covering the history and development is long; here we just mention some: \cite{CM:book, szasz:challenges, golse:survey, dettmann:survey}. Beyond periodic media, statistical properties of the Lorentz gases for random configurations of scatterers have also been obtained (see e.g. \cite{GL:recurrenceTubes, AL:random}). One of the biggests gaps in our knowledge of Lorentz gases is the case of aperiodic media of low complexity, as stated in \cite{szasz:challenges, marklof:ICM}. The only rigorous work to date is the work of Marklof-Str\"ombergsson \cite{MS:quasi}, which derives limit theorems for the so-called Boltzman-grad limit for quasicrystals, that is, studies of the low density limit in aperiodic media with long range interactions. 

      This is the first paper in a program to study the statistical properties of Lorentz gases in general aperiodic repetitive media of low complexity which aims to close this gap. Our goal is to be broad in the sense of not assuming much beyond aperiodicity, repetitivity and finite local complexity. In particular we do not assume that the medium has long-range order (i.e. is not necessarily a quasicrystal). In this paper we further assume that the Lorentz gas 
belongs to category A in the classification of Chernov-Markarian \cite{CM:book}, meaning that it has finite horizon and no corners.
The assumption on finite horizon indicates that we treat cases far from the zero-density limit and therefore our setting is different from \cite{MS:quasi}. 
The novelty of our approach is the use of tiling spaces to obtain a compact model of the aperiodic Lorentz gas on the plane.
In the case of a periodic Lorentz gas the corresponding compact model is given by a Sinai billiard and is a uniformly hyperbolic diffeomorphism with singularities.  
Unlike Sinai billiards, 
the compact model we obtain by using tiling spaces is not defined on a manifold (our space locally is a product of a Euclidean ball and a Cantor set), hence cannot be thought of as a diffeomorphism. On the other hand, our model exhibits essential features common to partially hyperbolic systems.
This makes the problem qualitatively different from the periodic case.
      \subsection{Statement of results}
      Before stating our results, we describe the precise setting in which we work. A \textbf{tiling} $\mathcal{T}$ of $\mathbb{R}^d$ is a collection of compact, connected sets, called \textbf{tiles}, such that their union covers $\mathbb{R}^d$ and such that any two distinct tiles in the tiling may only intersect at their boundaries (see \S \ref{sec:background} for background on tilings). We refer to a finite union $\mathcal{P}$ of tiles of $\mathcal{T}$ as a \textbf{patch} of $\mathcal{T}$ (see Figure \ref{fig:HalfHexesPatch} for an example). A tiling is \textbf{repetitive} if given any patch $\mathcal{P}$ of $\mathcal{T}$ there exists an $R>0$ such that any ball of radius $R$ in $\mathbb{R}^d$ contains a copy of $\mathcal{P}$ in the portion of $\mathcal{T}$ inside this ball. It has \textbf{finite local complexity} if for any $R>0$ there is a finite list of patches of diameter at most $R$ that occur in the tiling.

      Given a repetitive, aperiodic tiling $\mathcal{T}$ of finite local complexity, there is a compact metric space $\Omega$ called the \textbf{tiling space of $\mathcal{T}$} which contains all tilings which are locally indistinguishable from translations of $\mathcal{T}$. Assuming that $\mathcal{T}$ is repetitive and has finite local complexity, it is one of the continuum of tilings in $\Omega$. The space $\Omega$ is not a manifold but has the local product structure of an open Euclidean ball and a Cantor set. There is a minimal action of $\mathbb{R}^d$ on this space, making $\Omega$ a foliated space, foliated by orbits of this action.
      
\begin{figure}[t]
  \centering
  \includegraphics[width = 5.5in]{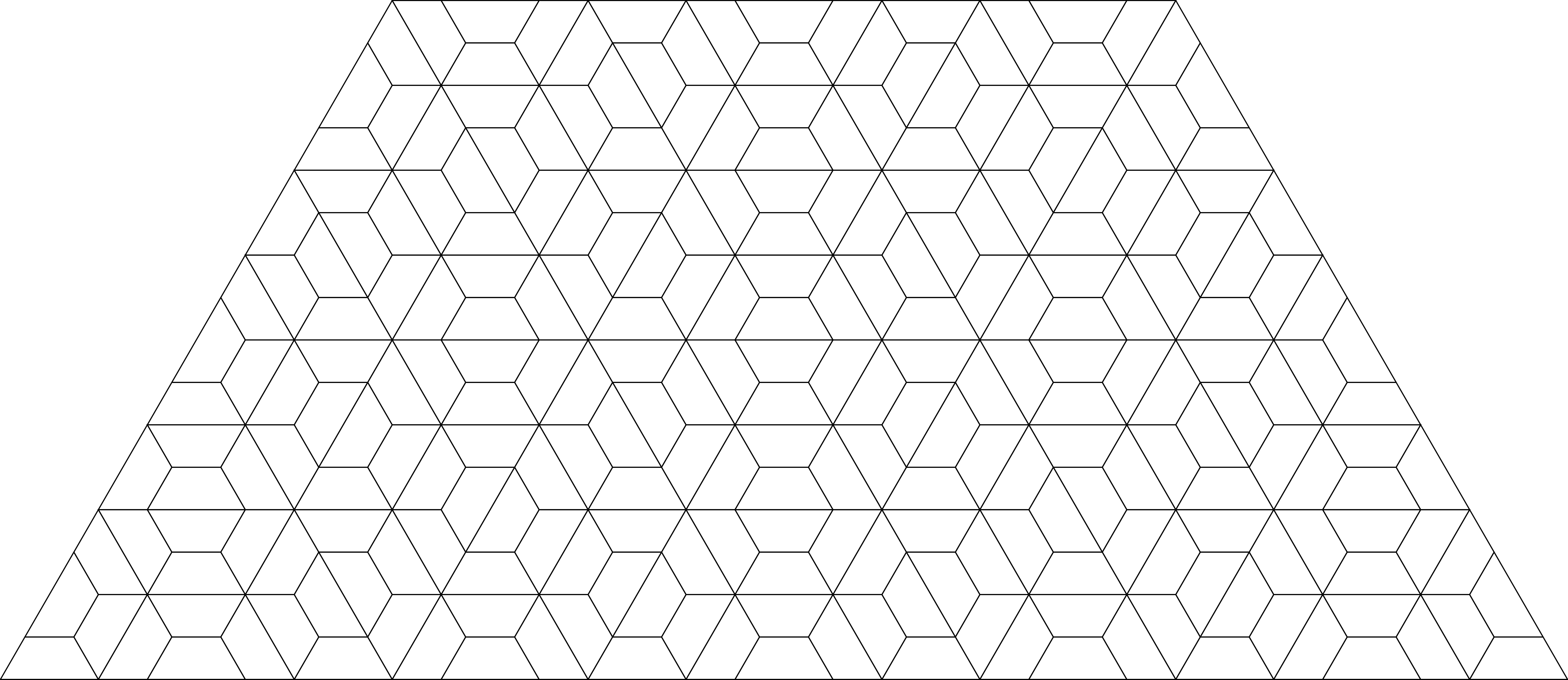}
  \caption{A patch of the half-hex aperiodic tiling.}
  \label{fig:HalfHexesPatch}
\end{figure}
      
If $\mathcal{T}$ is a tiling of $\mathbb{R}^d$, a function $g:\mathbb{R}^d\rightarrow \mathbb{R}$ is \textbf{$\mathcal{T}$-equivariant} (or pattern-equivariant) if there exists an $R>0$ such that if for some $x_1,x_2\in\mathbb{R}^d$, the largest patches $\mathcal{P}_1$, $\mathcal{P}_2$ contained in $B_R(x_1)$ and $B_R(x_2)$, respectively, are translation equivalent (that is, $\mathcal{P}_1$ and $\mathcal{P}_2$ are translated copies of each other), then $g(x_1) = g(x_2)$. In other words, the value of $g(x)$ depends only on the local pattern around $x$. 
More generally, if $X$ is any set and $\mathcal{T}$ is a tiling of $\mathbb{R}^d$, a function $g:\mathbb{R}^d\times X\rightarrow \mathbb{R}$ is said to be \textbf{$\mathcal{T}$-equivariant} if 
for any $z\in X$ the function $g_z(x):=g(x,z)$ is $\mathcal{T}$-equivariant.

Pattern equivariant functions are in a sense the most natural class of functions to consider because they are closely related to functions on the tiling space: a function on $\mathbb{R}^d$ is $\mathcal{T}$ equivariant if and only if there is a transversally locally constant (see \S \ref{subsec:fun} for more details) function $h:\Omega\rightarrow \mathbb{R}$ such that $g(v) = h\circ \varphi_v(\mathcal{T})$, where $\varphi$ is the translation action on the tiling space $\Omega$.

     For a collection $\mathcal{S}\subset \mathbb{R}^d$ of topological balls with $C^3$ boundaries, 
called \textbf{scatterers}, the Lorentz gas is the system defined for initial conditions in $(\mathbb{R}^d\setminus \overline{\mathcal{S}})\times S^{d-1}_+$ as the free flight of a particle in $\mathbb{R}^d\setminus \overline{\mathcal{S}}$ in the direction determined by a point in $S^{d-1}_+$ and bouncing elastically off $\partial \mathcal{S}$. This system has \textbf{finite horizon} if there is a global constant $M>0$ such that the time between any two collisions with $\partial \mathcal{S}$ is bounded above by $M$;
it has \textbf{no corners} if any two scatterers are separated;
 it is \textbf{dispersive} if every connected component of $\overline{\mathcal{S}}$ (that is, a specific scatterer) is strictly convex.
We refer to a collection of dispersive scatterers (and to the corresponding Lorentz gas) of finite horizon and with no corners as being of \textbf{category A}.
     
Given a tiling $\mathcal{T}$ of $\mathbb{R}^d$, a collection of scatters $\mathcal{S}$ is \textbf{$\mathcal{T}$-equivariant} if there exists an $R>0$ such that if $x_1,x_2\in\mathbb{R}^d$ are the centers of mass of two scatterers $S_1,S_2\subset \mathcal{S}$ and the largest patches $\mathcal{P}_1$, $\mathcal{P}_2$ contained in $B_R(x_1)$ and $B_R(x_2)$, respectively, are translation equivalent, then the scatterers $S_1$ and $S_2$ are translation equivalent. In other words, the geometry of a scatterer and relative position of scatterers surrounding it in a $\mathcal{T}$-equivariant collection of scatterers depends only on the local pattern around it. An example of this would be the collection of scaterrers centered at the vertex of every tile in a tiling. Another example is depicted in Figure \ref{fig:HalfHexes}.
Given a $\mathcal{T}$-equivariant collection of scatterers $\mathcal{S}\subset\mathbb{R}^d$, a function $g:\partial \mathcal{S}\rightarrow \mathbb{R}$ is $\mathcal{T}$-equivariant if it is the restriction to $\partial \mathcal{S}$ of a $\mathcal{T}$-equivariant function on $\mathbb{R}^d$.
     
\begin{figure}[t]
  \centering
  \includegraphics[width = 6.5in]{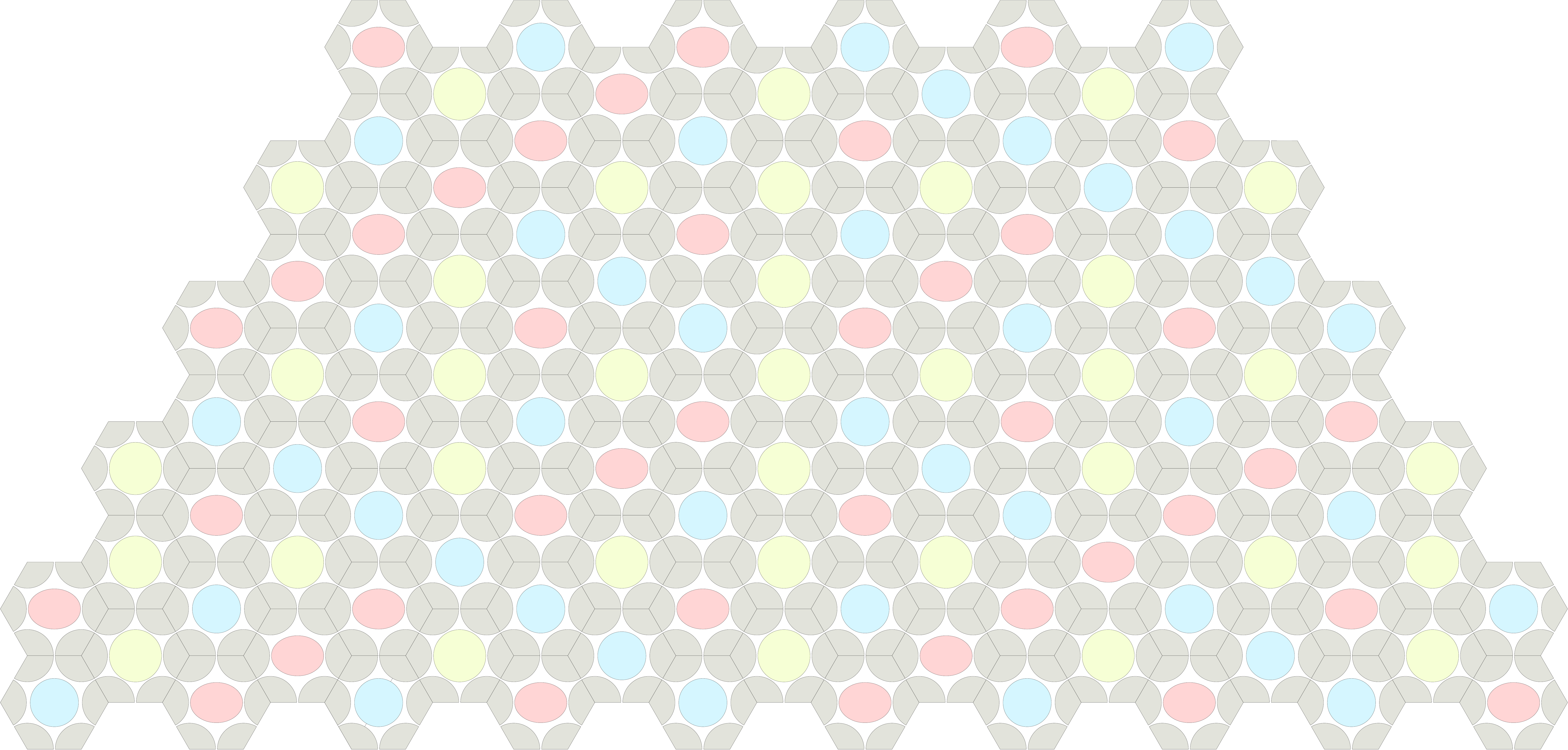}
  \caption{An aperiodic Lorentz gas configuration. The scatterers depend only on local configurations of the half-hex tiling in Figure \ref{fig:HalfHexesPatch}.}
  \label{fig:HalfHexes}
\end{figure}

Tilings as coverings of $\mathbb{R}^d$ have their discrete counterpart called \textbf{Delone multisets}. Roughly speaking, if one thinks of a tiling as a set of tiles, one can think of the corresponding Delone multiset as a set of centers of masses of those tiles.
We give a precise definition of a Delone multiset and related notions in Section \ref{subsec:delone}. Even though we state our main results in the language of tilings, it is sometimes convenient to think of the corresponding Delone multisets instead.
In particular, given a Delone multiset one can easily define the collection of scatterers by taking each point in a Delone multiset as a center of mass of a scatterer (see Section \ref{subsec:delonescat}). 
     
      Given a repetitive, aperiodic tiling $\mathcal{T}$ of finite local complexity and a $\mathcal{T}$-equivariant collection of scatterers $\mathcal{S}$, we can take $\partial\mathcal{S}\times S^{d-1}_+$ as a Poincar\'e section for the Lorentz gas. This set can be given a metric topology which naturally identifies it as a subset of $\mathcal{D} = \partial\Sigma\times S^{d-1}_+\subset \Omega\times S^{d-1}_+$. The set $\mathcal{D}$ is then the Poincar\'e section to the continuum of Lorentz gases defined in all tilings in $\Omega$ as long as we use the same local dependence on scatterers which defined $\mathcal{S}$. The set $\mathcal{D}$ is a compact metric space where a global Poincar\'e map $f:\mathcal{D}\rightarrow \mathcal{D}$ is defined. Given any $\mathbb{R}^d$-invariant measure on $\Omega$, there is a naturally derived measure $\bar{\mu}$ on $\mathcal{D}$ which is $f$-invariant. Our first result is for the planar Lorentz gas on aperiodic tilings (see Theorem \ref{thm:Kproperty} in \S \ref{sec:K}).

 \begin{theorem}
        \label{thm:K}
        Let $\mathcal{T}$ be a repetitive, aperiodic tiling of $\mathbb{R}^2$ of finite local complexity, $\mathcal{S}$ be a $\mathcal{T}$-equivariant collection of dispersive scatterers with $C^3$ boundary
defining a Lorentz gas of category A,
 and $\mu$ an $\mathbb{R}^2$-invariant ergodic probability measure on $\Omega_\mathcal{T}$. Then the map $f : \mathcal{D} \to \mathcal{D}$ has the K property with respect to $\bar{\mu}$.
      \end{theorem}
      \begin{remark}
        Theorem \ref{thm:K} brings an entire new class of examples of K systems to the theory of dynamical systems. We do not know under which conditions these systems are Bernoulli. We are currently pursuing this question. 
      \end{remark}
      We briefly outline the proof of this theorem. 
We consider measurable partitions of $\mathcal{D}$ into local stable and unstable manifolds respectively.
The celebrated fundamental theorem of Sinai 
for dispersive planar billiards holds in our case and provides a good control over the sizes of local stable and unstable manifolds. 
Following arguments identical to the proof of local ergodicity for periodic Sinai billiards, we obtain that Lebesgue-almost every two points on the collision space of a single scatterer can be connected by a stable-untable path.
It follows that 
$\bar{\mu}$-almost every connected component (a single scatterer along with all of its incoming/outgoing collisions) is fully contained (up to a subset of Lebesgue measure zero) in a single element of the Pinsker partition. 
Using the (minimal) foliated structure of the tiling space and the dynamics of $f$, it will follow that a dense union of connected components is contained in an element of the Pinsker partition. Finally, using again the structure of the tiling space and the ergodicity of $\mu$ (with respect to the translation $\mathbb{R}^2$-action) we will show that the Pinsker partition is trivial.
            
      From this mixing result on the compact system $(f, \mathcal{D},\bar{\mu})$, we obtain a mixing result for planar Lorentz gases of category A.
That is for a  $\mathcal{T}$-equivariant collection of scatterers $\mathcal{S} \subset \mathbb{R}^2$,  $\theta\in (-\pi/2,\pi/2)$, and $n\in\mathbb{N}$
we denote by
    $F^n_\theta:\partial\mathcal{S}\rightarrow \partial\mathcal{S}$ the map given by the first $n$ colisions of the planar Lorentz gas with initial angle $\theta\in (-\pi/2,\pi/2)$. 
The map $F^n_\theta$ is not defined on all of $\partial\mathcal{S}\times (-\pi/2,\pi/2)$ due to discontinuities, but this is a set of Lebesgue measure zero for all $n\in\mathbb{N}$.

 Let $g:\partial \mathcal{S}\times (-\pi/2,\pi/2)\rightarrow \mathbb{R}$ be a locally integrable, $\mathcal{T}$-equivariant function. 
We define the \textbf{planar average} of $g$ as the following limit 
      $$\beta(g):= \lim_{T\rightarrow \infty}\frac{1}{\mbox{Vol}(B_T)}\int_{-\pi/2}^{\pi/2}\int_{B_T(z)\cap \partial \mathcal{S}} g\, d\rho\, \cos\theta\, d\theta,$$
where $\rho$ denotes the arc-length and $\theta \in  (-\pi/2,\pi/2)$ is the angle. 
By Lemma \ref{lem:bump}, for every locally integrable, $\mathcal{T}$-equivariant function $g$ the above limit exists and does not depend on $z$.
We denote the set of $\mathcal{T}$-equivariant, locally integrable functions $g:\partial \mathcal{S}\times (-\pi/2,\pi/2)\rightarrow \mathbb{R}$ by $\Delta_{\mathcal{T}}(\mathcal{S})$.\\

   Let $\Omega$ be the tiling space of an aperiodic, repetitive tiling $\mathcal{T}'$ of $\mathbb{R}^2$ of finite local complexity, and let $\mu$ be 
   an invariant ergodic probability measure for the $\mathbb{R}^2$ action. Let $\partial\Sigma\subset \Omega$ be a transversally locally constant set which defines for every $\mathcal{T}\in\Omega$ a $\mathcal{T}$-equivariant collection $\mathcal{S}_\mathcal{T}\subset \mathbb{R}^2$ of dispersive scatterers of Category A with $C^3$ boundary.

   \begin{theorem}[Planar averaged mixing; see Theorem \ref{thm:generalMixing} in \S\ref{sec:mixing}]
     \label{thm:mixing1}
     For $\mu$-almost  every $\mathcal{T}\in\Omega$, for any two locally integrable, $\mathcal{T}$-equivariant functions $g_1,g_2:\partial \mathcal{S}\times (-\pi/2,\pi/2)\rightarrow \mathbb{R}$ there exist transversally locally constant functions $h_1,h_2:\mathcal{D}\rightarrow \mathbb{R}$ such that for any $n\in\mathbb{N}$ and all $T$ large enough (depending on $n$), we have that for any $x\in\mathbb{R}^2$
     \begin{equation}
       \label{eqn:ApproxleafMixing1}
       \int_{-\pi/2}^{\pi/2} \int_{B_T(x)\cap \partial\mathcal{S}_\mathcal{T}} g_1\circ F^n_\theta\cdot g_2 \, d\rho\, \cos\theta\, d\theta = \mathrm{Vol}(B_T)\langle h_1\circ f^n,h_2\rangle_{L^2_{\bar{\mu}}(\mathcal{D})} +  o_n(T^2),
     \end{equation}
     from which we obtain the convergence
     \begin{equation}
       \label{eqn:leafMixing1}
       \lim_{n\rightarrow \infty} \lim_{T\rightarrow \infty}\frac{1}{\mathrm{Vol}(B_T)}\int_{-\pi/2}^{\pi/2} \int_{B_T(x)\cap \partial\mathcal{S}_\mathcal{T}} g_1\circ F_\theta^n\cdot g_2\, d\rho\, \cos\theta\, d\theta = \beta(g_1)\beta(g_2).
     \end{equation}
   \end{theorem}

   If the action of $\mathbb{R}^2$ on $\Omega$ is uniquely ergodic, the above result can be strengthen to obtain \eqref{eqn:ApproxleafMixing1} and \eqref{eqn:leafMixing1} for every, and not just $\mu$-almost every $\mathcal{T}\in\Omega$. 
   Unique ergodicity  of the $\mathbb{R}^2$- action on the tiling space $\Omega_{\mathcal{T}'}$ generated by $\mathcal{T}'$  is equivalent to the tiling $\mathcal{T}'$ having \textbf{uniform patch frequency}. That is, the asymptotic frequency (\ref{eqn:frequency}) of a given patch in $\mathcal{T}$
being the same for every tiling $\mathcal{T}  \in \Omega_{\mathcal{T}'}$.

\begin{corollary}
   Let $\Omega$ be the tiling space of an aperiodic, repetitive tiling $\mathcal{T}'$ of $\mathbb{R}^2$ of finite local complexity and uniform patch frequency, and let $\mu$ be the unique measure invariant under the $\mathbb{R}^2$ action. Let $\partial\Sigma\subset \Omega$ be a transversally locally constant set which defines for every $\mathcal{T}\in\Omega$ a $\mathcal{T}$-equivariant collection $\mathcal{S}_\mathcal{T}\subset \mathbb{R}^2$ of dispersive scatterers of category A with $C^3$ boundary.
        
        For \textbf{every} $\mathcal{T}\in\Omega$, for any two locally integrable, $\mathcal{T}$-equivariant functions $g_1,g_2:\mathbb{R}^2\times (-\pi/2,\pi/2)\rightarrow \mathbb{R}$ there exist transversally locally constant functions $h_1,h_2:\mathcal{D}\rightarrow \mathbb{R}$ such that for any $n\in\mathbb{N}$ and all $T$ large enough (depending on $n$), we have that for any $x\in\mathbb{R}^2$
        \begin{equation}
          \int_{-\pi/2}^{\pi/2} \int_{B_T(x)\cap \partial\mathcal{S}_\mathcal{T}} g_1\circ F^n_\theta\cdot g_2 \, d\rho\, \cos\theta\, d\theta = \mathrm{Vol}(B_T)\langle h_1\circ f^n,h_2\rangle_{L^2_{\bar{\mu}}(\mathcal{D})} +  o_n(T^2),
        \end{equation}
        from which we obtain the convergence
        \begin{equation}
          \lim_{n\rightarrow \infty} \lim_{T\rightarrow \infty}\frac{1}{\mathrm{Vol}(B_T)}\int_{-\pi/2}^{\pi/2} \int_{B_T(x)\cap \partial\mathcal{S}_\mathcal{T}} g_1\circ F_\theta^n\cdot g_2\, dt\, \cos\theta\, d\theta = \beta(g_1)\beta(g_2).
        \end{equation}
\end{corollary}

The estimate on the error term in Theorem \ref{thm:mixing1} can be sharpened in the case of substitution tilings (also called self-similar tilings) which we introduce in Section \ref{sec:background} (see Definition \ref{def:substitution}). Namely,  let $\Omega$ be the tiling space of a self-similar tiling $\mathcal{T}'$ of $\mathbb{R}^2$ of finite local complexity, and let $\mu$ be the unique measure invariant under the $\mathbb{R}^2$ action. Let $\partial\Sigma\subset \Omega$ be a transversally locally constant set which defines for every $\mathcal{T}\in\Omega$ a $\mathcal{T}$-equivariant collection $\mathcal{S}_\mathcal{T}\subset \mathbb{R}^2$ of dispersive scatterers of category A with $C^3$ boundaries.
We have the following.

\begin{theorem}[See Theorem \ref{thm:SSmixing} in \S \ref{sec:mixing}]
  \label{thm:mixing2}
  For any $n$ there exist $d^+\geq 1$ functions $C_i^n:\Delta_\mathcal{T}(\mathcal{S})\times \Delta_\mathcal{T}(\mathcal{S})\rightarrow \mathbb{R}$ (where $d^+$ depends on the spectrum of the substitution matrix) such that for any $\mathcal{T}\in\Omega$, for $g_1,g_2\in \Delta_\mathcal{T}(\mathcal{S}_\mathcal{T})$ there exist transversally locally constant functions $h_1,h_2:\mathcal{D}\rightarrow \mathbb{R}$ such that for any $n\in\mathbb{N}$ and all $T$ large enough (depending on $n$), we have that
  \begin{equation}
    \label{eqn:SSapproxMixing1}
    \begin{split}
      &\int_{-\pi/2}^{\pi/2} \int_{B_T\cap \partial\mathcal{S}_\mathcal{T}} g_1\circ F^n_\theta\cdot g_2 \, d\rho\, \cos\theta\, d\theta\\
      &\hspace{.5in}= \sum_{i=1}^{d^+}L_i(T)T^{2\frac{\log |\lambda_i|}{\log\lambda_1}} C_i^n(g_1,g_2) +  \mathcal{O}_n(T),\\
      &\hspace{.5in}= \mathrm{Vol}(B_T)\langle h_1\circ f^n,h_2\rangle_{L^2_{\bar{\mu}}(\mathcal{D})} + \sum_{i=2}^{d^+} L_i(T)T^{2\frac{\log |\lambda_i|}{\log\lambda_1}} C_i^n(g_1,g_2) +  \mathcal{O}_n(T),
    \end{split}
  \end{equation}
  where $L_i(T) /\log(T)^{s-1}\in \mathcal{O}(1)$ if the size of the Jordan block associated to the $i^{th}$ eigenvalue of the substitution matrix is $s$. Moreover, the planar averaged mixing (\ref{eqn:leafMixing1}) also holds.
\end{theorem}
\begin{remark}
  \label{rem:1}
  Some remarks:
  \begin{enumerate}
  \item These results hold independent of whether the tiling models a quasicrystal, that is, regardless of whether the tiling $\mathcal{T}$ has long range order, thus applying to systems which do not come from cut-and-project constructions. In particular, any non-Pisot substitution tiling does not have long range order \cite{solomyak:SS} but enjoys an asymptotic expansion as in (\ref{eqn:SSapproxMixing1}).
  \item The class of substitution tilings is not the only class for which detailed asymptotics as in (\ref{eqn:SSapproxMixing1}) can be given. There is a recently-developed theory of (globally) random substitution tilings \cite{ST:random, T:TTT} which produce tilings by the random composition of different substitution rules, resulting in repetitive, aperiodic tilings of finite local complexity (and which includes the self-similar case). The detailed expansion in (\ref{eqn:SSapproxMixing1}) comes from rates of deviations of ergodic integrals for the $\mathbb{R}^2$ action on the tiling space for self-similar tilings \cite{bufetov-solomyak:tilings}, which are generalized in \cite{T:TTT} for random substitution tilings\footnote{There are other works which study deviations of ergodic integrals. However, these make use of cohomology of the tiling space, which cannot be used here as the functions which we integrate have discontinuities.}. However, the necessary background needed to introduce the construction of these tilings falls outside the scope of this paper.
  \item We remark that our mixing results give new results for \emph{periodic} Lorentz gases in that the scatterer configuration can be periodic but the observables are not. As an example, suppose that all the scatterers in the example in Figure \ref{fig:HalfHexes} are modified to be spheres of the same radius (large enough to have finite horizon). This would give a periodic scatterer configuration. However, using observables which are $\mathcal{T}$-equivariant with respect to the half-hex tiling in Figure \ref{fig:HalfHexesPatch} and which satisfy the conditions of our theorems, we also obtain planar mixing for this Lorentz gas with periodic scatterer configuration.
  \item The class of pattern-equivariant, locally integrable functions is the largest class we know for which the detailed expansion (\ref{eqn:SSapproxMixing1}) holds. However, it is not the largest class for which the expansion (\ref{eqn:ApproxleafMixing1}) holds. More specifically, the class of pattern-equivariant functions considered here agrees with the class of \textbf{strongly} pattern equivariant functions considered in \cite{kellendonk:PEF}. The mixing equations (\ref{eqn:ApproxleafMixing1}) and (\ref{eqn:leafMixing1}) also hold for \textbf{weakly} pattern equivariant functions (see again \cite{kellendonk:PEF}), and for certain limits of them. We focus on (strongly) pattern-equivariant functions here due to their evident relation to the aperiodic tiling.
  \item The averages computed using balls $B_T$ do not really need to be done with balls. In fact, any good sequence of averaging sets, such as cubes, F{\o}lner, or Van Hove sequences work just as well. This type of mixing is similar to the \emph{global-global mixing} of Lenci \cite{lenci:infMixing}.
  \item We do not know whether the Lorentz gases which we consider are \emph{recurrent}, so we do not know whether they belong to the class of aperiodic Lorentz gases studied by Lenci in \cite{lenci:recurrent}. Although we study the Lorentz gas via a compact factor with an invariant probability measure for which Poincar\'e recurrence holds, this does not imply the recurrence of the system on $\mathbb{R}^2$.
  \end{enumerate}
\end{remark}
The result on the K property implies the ergodicity of the billiard map $f: \mathcal{D} \to \mathcal{D}$, which in turn gives a result on ergodicity of the Lorentz gas flow $\bar{\phi}_t$ defined on the tiling space $\Omega_{\mathcal{T}}$. 
As such, for any $\mathcal{T}'\in\Omega_{\mathcal{T}}$ we can consider the corresponding Lorentz gas flow $\phi_t^{\mathcal{T}'}$ defined on the plane by the scatterer configuration $\mathcal{S}'$ associated to $\mathcal{T}'$.  

For any $z\in\partial\mathcal{S}'$ and $\theta\in(-\pi/2,\pi/2)$, let $\tau(z,\theta)\in(0,M]$ denote the free flight time of the Lorentz gas starting at $z$ with direction $\theta$. 
We denote by
$$  \bar{\tau}:=  \int_{\mathcal{D}}\tau(\mathcal{T}', \theta)\,d\bar{\mu}  $$
the mean free flight time on $\mathcal{D}$. For a function $f:\mathbb{R}^2\times (-\pi/2,\pi/2) \rightarrow \mathbb{R}$ define
\begin{equation}
  \label{eqn:segmentNorm}
  \|f\|_{(1)}: = \sup_{\theta\in \left(-\frac{\pi}{2},\frac{\pi}{2}\right)}\sup \left\{\frac{\int_{\gamma_\theta} |f(\cdot, \theta)|}{\mbox{length}(\gamma_\theta)}:\gamma_\theta\mbox{ is a finite length geodesic in $\mathbb{R}^2$ in direction $\theta$}\right\}
\end{equation}
and for $\mathcal{T}\in\Omega$, the space
\begin{equation}
  \label{eqn:goodSpace}
  \mathfrak{G}_\mathcal{T}:= \left\{f: f\mbox{ is a locally integrable, }\mathcal{T}\mbox{-equivariant on $\mathbb{R}^2\times (-\pi/2,\pi/2)$ with }\|f\|_{(1)}<\infty \right\}.
\end{equation}
For any  $g\in\mathfrak{G}_{\mathcal{T}'}$ the quantity
\begin{equation}
  \label{eqn:segmentAvg}
  \bar{G}(z,\theta) = \int_0^{\tau(z,\theta)}g\circ \phi^{\mathcal{T}'}_t(z,\theta)\,dt,
\end{equation}
is well defined, and it is a $\mathcal{T}'$-equivariant function on $\partial\mathcal{S}'\times (-\pi/2,\pi/2)$. We prove the following in Section \S \ref{sec:ergFlow}.
\begin{theorem}[Ergodicity of the Lorentz gas flow]
  \label{thm:erg}
  Let $\mathcal{T}$ be a repetitive, aperiodic tiling of $\mathbb{R}^2$ of finite local complexity, $\mathcal{S}$ be a $\mathcal{T}$-equivariant collection of scatterers of category A with $C^3$ boundary and $\mu$ an $\mathbb{R}^2$-invariant ergodic probability measure on $\Omega_\mathcal{T}$.    
  Let $h:\Omega\times (-\pi/2,\pi/2)\rightarrow \mathbb{R}$ be a $L^1$ transversally locally constant function which defines for every $\mathcal{T}'\in\Omega$ a $ g_{\mathcal{T'}}\in\mathfrak{G}_{\mathcal{T}'}$. For $\bar{\mu}$-almost every $(\mathcal{T}',\theta)\in\partial \Sigma\times [-\pi/2,\pi/2]$
  and $T>0$ we have that
  $$\frac{1}{T}\int_0^Tg_{\mathcal{T}'}\circ \phi^{\mathcal{T}'}_t(0,\theta)\, dt = \bar{\tau}^{-1} \frac{1}{\mathrm{Vol}(B_T)}\int_{-\pi/2}^{\pi/2}\int_{B_T\cap \partial\mathcal{S}'}\bar{G}_{\mathcal{T}'}\, d\rho \,  \cos\theta\, d\theta + o(T),$$
  as well as
  \begin{equation}
    \label{eqn:ergodicity}
    \lim_{T\rightarrow \infty}\frac{1}{T}\int_0^Tg_{\mathcal{T}'}\circ \phi^{\mathcal{T}'}_t(0,\theta)\, dt =    \bar{\tau}^{-1}    \beta(\bar{G}_{\mathcal{T}'})  .
  \end{equation}
\end{theorem}
We note that the last four items in Remark \ref{rem:1} also apply to this ergodicity theorem.
      
This paper is organized as follows. Section \ref{sec:background} covers the necessary background material from the field of aperiodic tilings and aperiodic point (Delone) sets. In  Section \ref{sec:gases} we make rigorous the definition of aperiodic scatterer configuration under which we work, and we express it both as defined by an aperiodic tiling or by an aperiodic Delone multiset. We define the aperiodic Lorentz gas and its compact model, which generalizes the Sinai billiard, and identify the canonical invariant measure for this system. Section \ref{sec:womanifolds} collects basic facts about stable and unstable manifolds for our compact system, most of which follow from the theory of Sinai billiards. In Section \ref{sec:K} we prove the K property for our compact factor. Sections \ref{sec:mixing} and \ref{sec:ergFlow} prove the mixing (Theorems \ref{thm:mixing1} and \ref{thm:mixing2}) and ergodicity results (Theorem \ref{thm:erg}), respectively.
      \begin{ack}
        We thank F. Rodriguez-Hertz for fruitful discussions on K-systems, as well as D. Dolgopyat for enlightening discussions on Lorentz gases. We thank the tornadoes for all the good times.
      \end{ack}
      \section{Tilings and tiling spaces}
      \label{sec:background}
      Here we cover the necessary backgroung on tilings. For a more in-depth treatment, see \cite{robinson:survey, sadun:book,BG:book}.
      
      A \textbf{tiling} $\mathcal{T}$ of $\mathbb{R}^d$ is a covering of $\mathbb{R}^d$ by compact sets $t$, called \textbf{tiles}, such that any two tiles of $\mathcal{T}$ may only intersect along their boundaries. The translation $\varphi_v(\mathcal{T})$ of a tiling $\mathcal{T}$ is the tiling obtained by translating each of the tiles $t$ of $\mathcal{T}$ by the vector $- v \in\mathbb{R}^d$. If there exist compact sets $t_1,\dots, t_M$ such that every tile $t\in\mathcal{T}$ is translation-equivalent to some $t_i$, then we call the collection $\{t_1,\dots, t_M\}$ the \textbf{prototiles} of $\mathcal{T}$. A union $\mathcal{P}\subset \mathcal{T}$ of a finite number of tiles of $\mathcal{T}$ is called a \textbf{patch} of $\mathcal{T}$. For any bounded set $A\subset\mathbb{R}^d$ define the patch
$$\mathcal{O}^-_\mathcal{T}(A) = \mbox{ largest patch of $\mathcal{T}$ completely contained in $A$}.$$

For two tiles $t,t'\in\mathcal{T}$, let $t\sim t'$ denote translation-equivalence, that is, if there exists a $v \in\mathbb{R}^d$ such that $\varphi_v (t) = t'$, then one writes $t\sim t'$.

  The types of tilings which will be considered here have the following properties. They are:
  \begin{description}
  \item[Aperiodic] If $\mathcal{T}$ is a tiling and $\varphi_{v}(\mathcal{T}) = \mathcal{T}$, then $v = 0$.
  \item[Repetitive] For any patch $\mathcal{P}\subset \mathcal{T}$ there exists an $R>0$ such that every ball of radius at least $R$ around any point in $\mathbb{R}^d$ contains a patch which is translation-equivalent to $\mathcal{P}$.
  \item[Finite Local Complexity] For every $R>0$ there is a finite collection of patches
$\{\mathcal{P}_i^R\}$
 such that any patch found inside a ball of radius $R$ is translation equivalent to some $\mathcal{P}_j^R$.
  \end{description}


  
  There is a metric on the set $\{\varphi_v(\mathcal{T})\}_{v\in\mathbb{R}^d}$ of translates of $\mathcal{T}$. First, define
  $$\bar{d}(\mathcal{T},\varphi_v(\mathcal{T})) = \inf\left\{ \varepsilon>0: \mathcal{O}^-_\mathcal{T}(B_{1/\varepsilon}(0)) = \mathcal{O}^-_{\varphi_{v+w}\mathcal{T}}(B_{1/\varepsilon}(0))\mbox{ for some }\|w\|\leq \varepsilon\right\},$$
  and from this define
  \begin{equation}
    \label{eqn:metric}
    d(\mathcal{T},\varphi_v(\mathcal{T})) = \min\left\{1,\bar{d}(\mathcal{T},\varphi_v(\mathcal{T}))\right\}.
  \end{equation}
  That (\ref{eqn:metric}) is metric is a standard fact in the tilings literature. Denote by
  $$\Omega_\mathcal{T} := \overline{\{\varphi_v(\mathcal{T}):v\in\mathbb{R}^d\}}$$
  the metric completion of the set of all translates of $\mathcal{T}$ with respect to the metric (\ref{eqn:metric}). This set is called the \textbf{tiling space} of $\mathcal{T}$. Note that if $\mathcal{T}$ is repetitive and has finite local complexity then we have that $\Omega_{\mathcal{T}'} = \Omega_{\mathcal{T}}$ for any other $\mathcal{T}'\in\Omega_\mathcal{T}$. By chosing a marked point in the interior of each prototile $t_i$ it can be assumed that every tile in $\mathcal{T}$ has a marking which is translation equivalent to the marking on prototiles.
The \textbf{canonical transversal} of $\Omega_\mathcal{T}$ is the set
  \begin{equation}
    \label{eqn:transversal}
    \mho_\mathcal{T} := \{\mathcal{T}'\in\Omega_\mathcal{T}:\mbox{ the marked point in the tile containing the origin is the origin}\}.
  \end{equation}
  Let $\mathcal{P}\subset \mathcal{T}$ be a patch and $t\in\mathcal{P}$ a tile in the patch. Then
  \begin{equation}
    \label{eqn:transCyl}
    \mathcal{C}_{\mathcal{P},t} = \{\mathcal{T}'\in\Omega_\mathcal{T}: \mbox{ $\mathcal{P}$ is a patch in $\mathcal{T}'$ and the distinguished point in the tile $t$ is the origin} \}
  \end{equation}
  is the $(\mathcal{P},t)$-cylinder set, also called a \textbf{transversal cylinder set}. For $\varepsilon>0$ the $(\mathcal{P},t,\varepsilon)$-set is defined by
  \begin{equation}
    \label{eqn:Cyls}
    \mathcal{C}_{\mathcal{P},t,\varepsilon} := \bigcup_{\|v\|<\varepsilon}\left\{ \varphi_v(\mathcal{T'}):\mathcal{T}'\in  \mathcal{C}_{\mathcal{P},t}\right\}\subset \Omega_\mathcal{T}.
  \end{equation}
  The tiling space is compact if $\mathcal{T}$ has finite local complexity. Moreover, if $\mathcal{T}$ is a repetitive tiling of finite local complexity then the topology of $\Omega_\mathcal{T}$ is generated by cylinder sets of the form $\mathcal{C}_{\mathcal{P},t,\varepsilon}$ for arbitrarily small $\varepsilon>0$. This gives $\Omega_\mathcal{T}$ the local product structure of $B_\varepsilon\times \mathcal{C} $, where $\mathcal{C}$ is a Cantor set and $B_\varepsilon\subset \mathbb{R}^d$ is an open ball.

  Most of the well-known tilings, such as the Penrose tiling, can be constructed from a substitution rule, giving the tiling a self-similar hierarchical structure. This is done as follows.
  \begin{definition}\label{def:substitution}
    Given a set of prototiles $t_1,\dots, t_M$ each viewed as a subset of $\mathbb{R}^d$, an \textbf{expansion and substitution} rule is a way of expressing expanded copies of the prototiles as unions of copies of the prototiles. In other words, such a rule determines a $\lambda>1$ (the expansion) such that for any $1\leq i\leq M$,
    $$\lambda\cdot t_i = \bigcup_{j=1}^M\bigcup_{\ell=1}^{r(i,j)}t_j+v_{i,j,\ell},$$
    where $v_{i,j,\ell}\in\mathbb{R}^d$ are vectors and $r(i,j)$ is the number of tiles of type $j$ found in the expanded copy $\lambda t_i$ under this substitution rule. 
  \end{definition}
The matrix $\mathcal{R}$ with $\mathcal{R}_{ij} = r(i,j)$ is called the \textbf{substitution matrix} for the expansion and substitution. A substitution is \textbf{primitive} if the substitution matrix is a primitive matrix (i.e. $\mathcal{R}^n$ has all positive entries for all $n$ large enough).\\
  
  For a tiling $\mathcal{T}$, a patch $\mathcal{P}\subset \mathcal{T}$ and Borel set $A\subset \mathbb{R}^d$, let $L_{\mathcal{T}}(\mathcal{P},A) = \# \mbox{ of copies of $\mathcal{P}$ in $A$}$. The \textbf{asymptotic frequency} of a patch $\mathcal{P}$ in $\mathcal{T}$ is the quantity
  \begin{equation}
    \label{eqn:frequency}
    \mathrm{freq}_\mathcal{T}(\mathcal{P}):= \lim_{T\rightarrow\infty}\frac{L_{\mathcal{T}}(\mathcal{P},B_T)}{{\mathrm{Vol}(B_T)}}.
  \end{equation}

  There is a natural action of $\mathbb{R}^d$ on $\Omega_\mathcal{T}$ by translation and this action is minimal if $\mathcal{T}$ is repetitive.
Throughout the paper we denote by $\mu$ an invariant Borel probability measure for the $\mathbb{R}^d$-action. By the local product structure 
of $\Omega_\mathcal{T}$, the measure $\mu$ has a natural product structure $ \mbox{Leb}\times \nu $, where $\nu$ is a transverse measure to the action of $\mathbb{R}^d$ and invariant under holonomy. In this paper we will be concerned with the cases when
$\mu$ is ergodic.
Tiling spaces coming from primitive substitution rules, for example, have a uniquely ergodic action \cite{solomyak:SS}. 
In fact, a large class of tiling spaces constructed from random substitutions also give rise to uniquely ergodic systems. As such, the assumption that $\mu$ is an ergodic measure is not restrictive.
  
  Since open sets along the totally disconnected transversal $\mho_\mathcal{T}$ are given by cylinder sets of patches, the measure $\nu$ satisfies $\nu(\mathcal{C}_{\mathcal{P},t}) = \mbox{freq}_\mathcal{T}(\mathcal{P})$, where $\mbox{freq}_\mathcal{T}(\mathcal{P})$ is the asymptotic frequency of translation copies of the patch $\mathcal{P}$ in the tiling $\mathcal{T}$. If there is a unique invariant measure for the $\mathbb{R}^d$ action (i.e. the action is uniquely ergodic), then the patch frequencies \cite{solomyak:SS} are uniform, so even though frequencies are defined with respect to a tiling, in the uniquely ergodic case there is no dependence on the choice of a tiling $\mathcal{T}\in\Omega$, in which case we may write $\nu(\mathcal{C}_{\mathcal{P},t}) = \mbox{freq}(\mathcal{P})$.

\subsubsection{Approximants and inverse limits}
\label{subsec:approximants}
Tiling spaces are inverse limits. That is, there is a family of topological spaces $\{\Gamma_k\}_{k\in\mathbb{N}}$ (called the \textbf{approximants}) and maps $\gamma_k = \Gamma_k\rightarrow \Gamma_{k-1}$ such that there is a homeomoprhism
\begin{equation}
  \label{eqn:inverseLim}
  \Omega_\mathcal{T}\cong \lim_{\leftarrow} (\Gamma_k,\gamma_k) := \left\{\bar{z} = (z_1,z_2,\dots)\in\prod_{k\in\mathbb{N}}\Gamma_k:z_{k-1} = \gamma_k(z_k) \right\}.
\end{equation}
We denote by $\varpi_k:\Omega_\mathcal{T}\rightarrow \Gamma_k$ the canonical projections to the approximants defined by the inverse limit, and by $\pi_{\mathcal{T},k}:\mathbb{R}^d \rightarrow \Gamma_k$ the map $\pi_{\mathcal{T},k}(\tau) = \varpi_k\circ\varphi_\tau(\mathcal{T})$. If $\mathcal{T}$ has finite local complexity, each $\Gamma_k$ can be taken to be a compact space. See \cite[Chapter 2]{sadun:book} for a full discussion of tiling spaces as inverse limits. In \S \ref{subsec:apriodicScat} we will use the maps $\pi_{\mathcal{T},0}$ to define aperiodic scatterer configurations.

\begin{wrapfigure}{l}{0.25\linewidth}
  \centering
  \includegraphics[width = 1.75in]{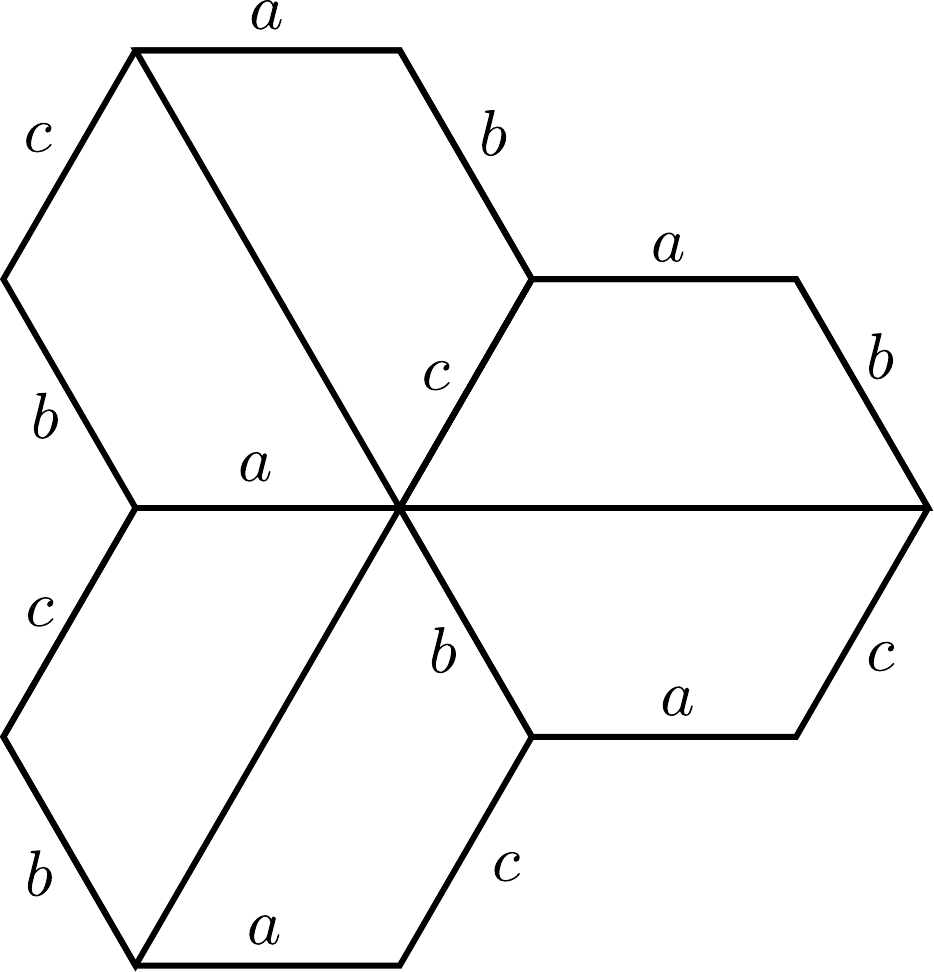}
  \caption{{\tiny Approximant for the half-hex tiling.}}
  \label{fig:HalfHexesApprox}
\end{wrapfigure}

To make this point of view more palatable, we remark that this point of view is most elegant when a tiling is constructed from a substitution rule, as first explained by Anderson and Putnam \cite{AP}. Figure \ref{fig:HalfHexesApprox}, for example, shows the approximant $\Gamma_k$ for the half-hex tiling in Figure \ref{fig:HalfHexesPatch} (edges labeled with the same symbol are identified). Because the tiling is obtained from a substitution rule (and thus a self-similarity), all approximants $\Gamma_k$ are the same. The maps $\gamma_k$ are given by substitution rule, that is, they subdivide each face of $\Gamma_k$ according to the substitution rule.
\subsection{Delone sets and multisets}\label{subsec:delone}
A \textbf{Delone set} $\Lambda\subset \mathbb{R}^d$ is a countable subset satisfying the two conditions
  \begin{description}
  \item[Uniform discreteness] There exists an $r>0$ such that for any distinct $x, y \in \Lambda$ we have that $|x-y|>r$.
  \item[Relative density] There is an $R>0$ such that $B_R(x)\cap \Lambda \neq \varnothing$ for any $x\in\mathbb{R}^d$.
  \end{description}
  A \textbf{Delone multiset} is a Delone set of the form $\Lambda = \bigcup_i^k\Lambda_i$, where each $\Lambda_i$ is a Delone set and the sets $\Lambda_i$ are mutually disjoint. We can think of a Delone multiset as a decorated Delone set. That is each point $x\in\Lambda$ in a Delone multiset comes ``colored'' by an index, so $x\in \Lambda$ has color $i$ if $x\in\Lambda_i$.

  Given a Delone multiset $\Lambda$, a \textbf{cluster} $\mathcal{P}\subset \Lambda$ is a finite subset of $\Lambda$. It is \textbf{repetitive} if for every cluster $\mathcal{P}\subset \Lambda$ there exists an $R>0$ such that $B_R(x)\cap \Lambda$ contains a translated copy of $\mathcal{P}$ for any $x\in\mathbb{R}^d$. A Delone multiset has \textbf{finite local complexity} if for every $R>0$ there exists a finite collection of clusters $\{\mathcal{P}_i^R\}$ such that 
any cluster found in a ball of radius $R$
is translation equivalent to some $\mathcal{P}_j^R$. 
  
  Delone sets and tilings represent two points of view on the same type of object. Given a tiling $\mathcal{T}$, one can construct a Delone multiset by taking the center of mass of each tile and coloring them by tile type. Conversely, from a Delone multiset one can consider the Voronoi tessellation given by this multiset, giving a tiling of $\mathbb{R}^d$. Properties such as repetitivity, aperiodicity and finite local complexity are preserved by these types of operations.

  Delone multisets are also translated by $\mathbb{R}^d$ and one can give the set of translates of $\Lambda$ a metric analogous to the metric (\ref{eqn:metric}) used for the set of translations of a tiling. The metric completion of this set is the \textbf{pattern space} $\Omega_\Lambda$ of $\Lambda$. Like a tiling space, a pattern space for a repetitive, aperiodic Delone multiset of finite local complexity is a foliated metric space, where the leaves correspond to distinct $\mathbb{R}^d$ orbits. 

  The \textbf{canonical transversal} $\mho_\Lambda\subset \Omega_\Lambda$ of a pattern space $\Omega_\Lambda$ is the subset
  $$\mho_\Lambda := \left\{\Lambda ' \in\Omega_\Lambda: \bar{0}\in\Lambda ' \right\}.$$
  If $\Lambda$ is aperiodic, repetitive and of finite local complexity, then $\mho_\Lambda\subset \Omega_\Lambda$ is a Cantor set. Given a cluster $\mathcal{P}\subset \Lambda$ with $\bar{0}\in\mathcal{P}$, the \textbf{transversal cylinder set} $\mathcal{C}_{\mathcal{P}}$ is defined in direct analogy to (\ref{eqn:transCyl}). By shifting such sets by small translations, in analogy to (\ref{eqn:Cyls}), one obtains cylinder sets $\mathcal{C}_{\mathcal{P},\varepsilon}$ which generate the topology of the pattern space $\Omega_\Lambda$. Just as with tiling spaces, under the assumptions of aperiodicity, repetitivity and finite local complexity, the pattern space has local product structure of a Euclidean ball and a Cantor set.

  Given a cluster $\mathcal{P}\subset \Lambda$, the frequency of $\mathcal{P}$ in $\Lambda$ is computed as in (\ref{eqn:frequency}). As in tiling spaces, this defines a family of transverse measures $\nu$ which is invariant under the holonomy of the $\mathbb{R}^d$ action. Furthermore, the measure which is locally the product $ \mathrm{Leb}\times\nu$ is invariant under the $\mathbb{R}^d$ action.
\subsection{Functions}
\label{subsec:fun}
Let $\mathcal{T}$ be a repetitive, aperiodic tiling of $\mathbb{R}^d$ of finite local complexity, and $\Lambda$ a Delone multiset with the same properties.
\begin{definition}
For any set $X$, a function $g:\mathbb{R}^d\times X\rightarrow \mathbb{R}$ is \textbf{$\mathcal{T}$-equivariant} (or \textbf{pattern-equivariant}) if there exists an $R>0$ such that if for any $x,y\in\mathbb{R}^d$ we have that $\mathcal{O}^-(B_R(x))$
(see \S \ref{sec:background} for definition) is translation equivalent to $\mathcal{O}^-(B_R(y))$, then $g(x,z) = g(y,z)$ for any $z\in X$. A $\Lambda$-equivariant function $g:\mathbb{R}^d\times X\rightarrow \mathbb{R}$ is similarly defined for Delone multisets.
\end{definition}
In short, these are functions which depend only on the local configuration of the tiling or Delone multiset used for reference. 
\begin{definition}
  Let $\Omega$ be the space associated to either $\mathcal{T}$ or $\Lambda$ and $X$ any set. A function $h:\Omega\times X\rightarrow \mathbb{R}$ is \textbf{transversally locally constant} if there exists an $R>0$ such that for any two $p,p'\in\Omega$ we have that $B_R(\bar{0})\cap p = B_R(\bar{0})\cap p' $ implies $h(p,z) = h(p',z)$ for any $z\in X$.
\end{definition}
In other words, transversally locally constant functions are locally constant along transversals. These are the functions with the highest ``regularity'' in the direction of the (totally disconnected) transversal. Pattern-equivariant functions and transversally locally constant functions are related in a nice way (see \cite[Proposition 22]{KP:RS}).
\begin{proposition}\label{prop:relatefunctions}
  For $  p \in\Omega$ a function $g:\mathbb{R}^d\times X\rightarrow \mathbb{R}$ $p$-equivariant if and only if there is a transversally locally constant function $h:\Omega\times X\rightarrow \mathbb{R}$ such that
  $$g(x,z) = h(\varphi_{x}(p),z)$$
  for all $x\in\mathbb{R}^d$ and $z\in X$.
\end{proposition}

As a direct consequence of Proposition \ref{prop:relatefunctions} we obtain the following.
\begin{corollary}\label{cor:relateintegrals}
If $\mu$ is an $\mathbb{R}^d$-invariant ergodic probability measure on $\Omega$ and $\kappa$ is any measure on $X$, then for $\mu$- almost every $p\in \Omega$ and for every $p$-equivariant function $g:\mathbb{R}^d\times X\rightarrow \mathbb{R}$ we have that

$$  \lim_{T \to \infty}  \frac{1}{\mbox{Vol}(B_T)}  \int_X   \int_{B_T}   g(x,z) d\mbox{Vol}(x) d\kappa(z)    =  \int_X   \int_{\Omega}  h(p',z)  d\mu(p')  d\kappa(z),     $$

where  $h:\Omega\times X\rightarrow \mathbb{R}$  is a transversally locally constant function satisfying   $g(x,z) = h(\varphi_{x}(p),z)$
  for all $x\in\mathbb{R}^d$ and $z\in X$.

\end{corollary}

There is a notion of \textbf{weakly} pattern equivariant functions with a similar characterization, although they will not be featured here. A function $g:\mathbb{R}^d\times X\rightarrow \mathbb{R}$ is weakly $p$-equivariant if there exists a transversally continuous function $h:\Omega\times X\rightarrow \mathbb{R}$ such that $g(x,z) = h(\varphi_{x}(p),z)$ for all $x\in\mathbb{R}^d$ and $z\in X$.

\section{Lorentz gases}
\label{sec:gases}
Let $\mathcal{S}\subset \mathbb{R}^d$ be a collection of open and convex topological disks with $C^3$ boundaries, called the \textbf{scatterers}. The Lorentz gas is the system describing the behavior of a particle moving along geodesics in $\mathbb{R}^d\backslash \bar{\mathcal{S}}$ and having elastic collisions at $\partial \mathcal{S}$. 
In this paper we restrict our attention to scatterer configurations of \textbf{Category A}, that is such that the time between any two collisions with $\partial \mathcal{S}$ is bounded above by some $M=M(\mathcal{S})>0$ (finite horizon) and that the closures of each two scatterers are disjoint (no corners).

 
\subsection{Aperiodic scatterer configurations}
\label{subsec:apriodicScat}
  \begin{definition}
    Let $\mathcal{S}\subset \mathbb{R}^d$ be a collection of scatterers. We call $\mathcal{S}$ an \textbf{aperiodic configuration} if there exists no nontrivial $v \in\mathbb{R}^d$ such that $\varphi_v (\mathcal{S}) = \mathcal{S}$.
  \end{definition}
  
  In this paper we consider two ways of constructing aperiodic scatterer configurations, one through aperiodic tilings and the other through aperiodic Delone multisets. In this subsection we describe both constructions and point out relations between them.
  \subsubsection{$\mathcal{T}$-equivariant scatterers}
  Throughout this section, $\mathcal{T}$ is a repetitive, aperiodic tiling of $\mathbb{R}^d$ of finite local complexity.
  \begin{definition}
    A collection of \textbf{$\mathcal{T}$-equivariant scatterers} is a scatterer configuration $\mathcal{S}$ which depends on the local configuration of $\mathcal{T}$. More precisely, there exists an $R>0$ such that if $S,S'\subset\mathcal{S}$ are connected components (scatterers) and $x,y$ are their respective centers of mass, then if $\mathcal{O}_\mathcal{T}^-(B_R(x)) = \mathcal{O}_\mathcal{T}^-(B_R(y)) + x-y$ then $S = \varphi_{-(x-y)}(S')$, that is, the scatterer around $x$ is translation equivalent to that around $y$.
  \end{definition}  
  If $\mathcal{T}$ is a hierarchical tiling, then by the inverse limit construction in \S \ref{subsec:approximants}, a collection $\mathcal{S}$ of scatterers is $\mathcal{T}$-equivariant, if there exists a $k\geq 0$ and an open set $\mathcal{S}_k\subset \Gamma_k $ such 
that $\mathcal{S} = \pi_{\mathcal{T},k}^{-1}(\mathcal{S}_{k})$. Any such set $\mathcal{S}_k$ which lifts to the scatterers of an aperiodic Lorentz gas is called the \textbf{set of model scatterers}. Figure \ref{fig:HalfHexesApprox} shows the model scatterers on the approximant $\Gamma_1$ which gives the pattern equivariant scatterers in Figure \ref{fig:HalfHexes}.

\begin{wrapfigure}{r}{0.25\linewidth}
  \centering
  \includegraphics[width = 1.75in]{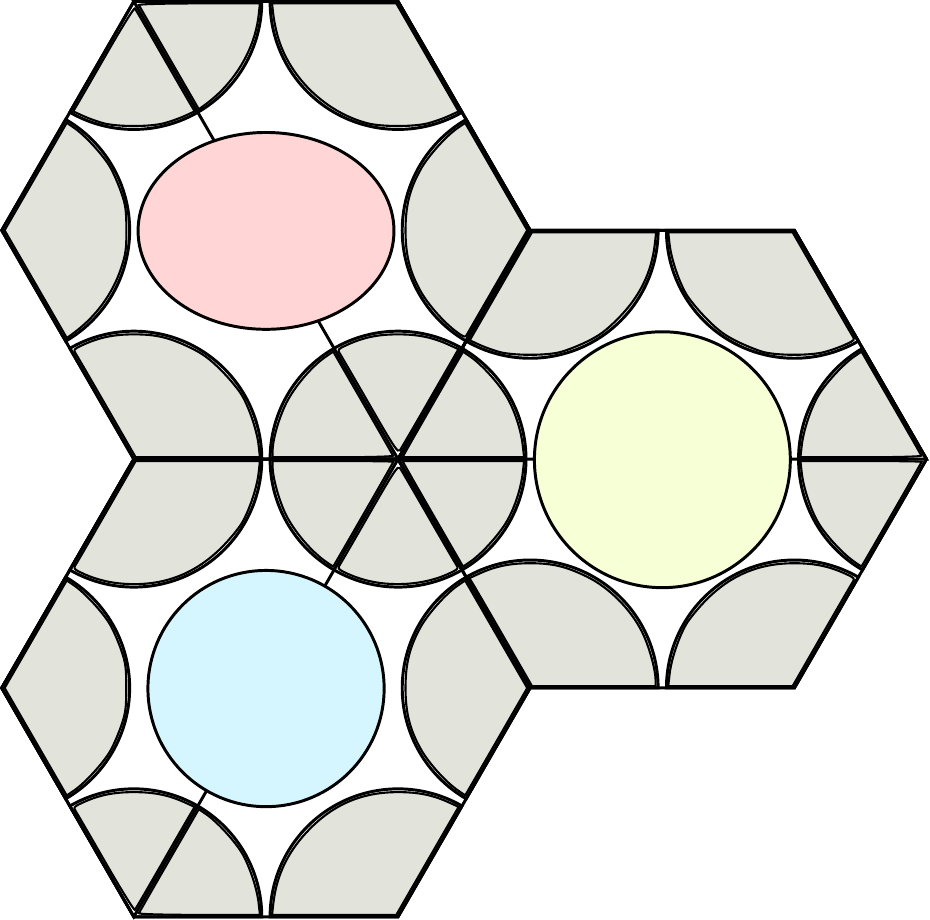}
  \caption{{\tiny Approximant and model scatterers.}}
  \label{fig:HalfHexesApprox2}
\end{wrapfigure}

\begin{lemma}
  \label{lem:level0}
  Without loss of generality we can assume that $k = 0$ for the set of model scatterers of a set of pattern-equivariant scatterers.
\end{lemma}
\begin{proof}
  Suppose that $\mathcal{S}_k\subset \Gamma_k$ is the set which defines the $\mathcal{T}$-equivariant collection of scatterers. Let $\Omega_\mathcal{T}$ be the tiling space of $\mathcal{T}$ defined as an inverse limit (\ref{eqn:inverseLim}). There is another tiling
space $\Omega '$ and a homeomorphism of tiling spaces $\Phi:\Omega_\mathcal{T}\rightarrow \Omega'$ as follows. For $\bar{z} = (z_0,z_1,z_2,\dots)\in\Omega_\mathcal{T}$ (in inverse limit coordinates), let $\Phi(z) = (z_k,z_{k+1},\dots)$, where $k$ is the index of the approximant where the set of model scatterers is defined. The map $\Phi$ has inverse $\Phi^{-1} = \gamma_1\circ \cdots\circ \gamma_k$. Thus, if the set of model scatterers is defined by the $k^{th}$ approximant in $\Omega_\mathcal{T}$, it is defined, up to scaling, by the $0^{th}$ approximant in $\Omega'$. 
\end{proof}
Given this observation, from now on we will assume that our set of scatterers is defined by a set $\mathcal{S}_0\subset \Gamma_0$ in the zeroth approximant.
\begin{remark}
  \label{lem:careful}
If $\mathcal{T}$ is a repetitive, aperiodic tiling of finite local complexity and $\mathcal{S}$ is a $\mathcal{T}$ equivariant collection of scatterers, then $\mathcal{S}$ is not necessarily aperiodic. This can be observed from Figures \ref{fig:HalfHexes} and \ref{fig:HalfHexesApprox2}. There is an underlying hexagonal (i.e. periodic) grid in the half hex tiling. If we were to pick the model scatterers to have the same constant curvature in Figure \ref{fig:HalfHexesApprox2}, they would lift to a periodic configuration in $\mathbb{R}^2$ (see Figure \ref{fig:HalfHexes}).
\end{remark}

\subsubsection{$\Lambda$-equivariant scatterers}\label{subsec:delonescat}

Let $\mathcal{S}\subset \mathbb{R}^d$ be a scatterer configuration, and $\mathfrak{C}(\mathcal{S})$ the set of connected components of $\mathcal{S}$. Let $\overline{\mathfrak{C}(\mathcal{S})} = \mathfrak{C}(\mathcal{S})/\sim$, where $S \sim S'$ if and only if $S = \varphi_v(S')$ for some $v\in\mathbb{R}^d$. The set $\overline{\mathfrak{C}(\mathcal{S})}$ records the geometrically different types of scatterers found in the collection $\mathcal{S}$. For each $c\in \overline{\mathfrak{C}(\mathcal{S})}$ we set
$$\Lambda_c^\mathcal{S} = \bigcup_{S \in c}\{\mbox{center of mass of }S\}\hspace{.4in}\mbox{ and define } \hspace{.4in} \Lambda_\mathcal{S} = \bigcup_{c\in \overline{\mathfrak{C}(\mathcal{S})}} \Lambda_c^\mathcal{S},$$
which may or may not be a Delone multiset.
\begin{definition}
Let $\Lambda = \bigcup_{i\in I} \Lambda_i$ be a Delone multiset. A scatterer configuration $\mathcal{S}$ is $\Lambda$-equivariant if there are $|I|$ distinct $C^3$ convex (topological) $d-1$ spheres $\{S_i\}$
such that $\Lambda_i = \Lambda_{i}^\mathcal{S}$ and $\Lambda_\mathcal{S} = \Lambda$. 
\end{definition}

\begin{definition}
We say that a configuration of scatterers $\mathcal{S}$ is repetitive, has finite local complexity, and uniform patch frequency, respectively, if and only if $\Lambda_\mathcal{S}$ forms a Delone multiset with those properties.
\end{definition}
\begin{lemma}
     If $\mathcal{S}$ is a repetitive, aperiodic scatterer configuration with finite local complexity and uniform patch frequency then there is a repetitive, aperiodic tiling $\mathcal{T}$ of finite local complexity and uniform patch frequency such that $\mathcal{S}$ is a $\mathcal{T}$-equivariant scatterer configuration.
\end{lemma}
\begin{proof}
 Consider the Delone multiset $\Lambda_\mathcal{S}$ and consider the Voronoi tessellation $\mathcal{T}_\mathcal{S}$ defined by $\Lambda_\mathcal{S}$. Then the scatterer configuration depends inside of which tile type its center of mass is, so it is $\mathcal{T}_\mathcal{S}$-equivariant.  
\end{proof}
Let $\mathcal{S}\subset \mathbb{R}^d$ be a repetitive, aperiodic scatterer configuration of finite local complexity and let $\Omega = \Omega_\mathcal{S}$ be the pattern space associated to $\mathcal{S}$ (or to the corresponding Delone multiset $\Lambda_\mathcal{S}$). The set of scatterers $\mathcal{S}$ can be seen as a particular subset of $\Omega$ as follows. Consider the set of translates $\varphi_v(\mathcal{S})$ of the scatterer configuration $\mathcal{S}$ such that the origin is contained in a scatterer, that is, 
$$\Sigma_\mathcal{S}' := \{v\in\mathbb{R}^d: \bar{0}\in \varphi_v(\mathcal{S})\}.$$
This defines a subset $\Sigma_\mathcal{S}$ of the orbit of $\Lambda_\mathcal{S}$ in $\Omega$ by
$$\Sigma_\mathcal{S} = \bigcup_{v\in \Sigma_\mathcal{S}'} \varphi_v(\Lambda_\mathcal{S})$$
whose closure is the set $\Sigma:= \overline{\Sigma_\mathcal{S}}\subset\Omega$. The set $\Sigma$ now defines an aperiodic configuration of $\Lambda$-equivariant scatterers for \textbf{any} Delone multiset $\Lambda\in\Omega$, not just in the $\mathbb{R}^d$ orbit of $\Lambda_\mathcal{S}$. 
We denote by $\partial\Sigma$ the corresponding set of boundaries of scatterers in $\Sigma$. This set has the following 
local product structure
\begin{equation}
  \label{eqn:localCoord}
  \partial\Sigma = \bigsqcup_{i\in \overline{\mathfrak{C}(\mathcal{S})}} \partial S_i \times \mathcal{C}_i,
\end{equation}
where $\{S_i\}$ is a finite collection of convex topological $d-1$-spheres (model scatterers), and $\mathcal{C}_i$ is some Cantor set. As such, any $\Lambda\in\partial\Sigma$ corresponds to an aperiodic scatterer configuration $\mathcal{S}_\Lambda \subset \mathbb{R}^d$ with the property that $\bar{0}\in \partial\mathcal{S}_\Lambda$.

Let us explain how one can think of the Cantor sets appearing in (\ref{eqn:localCoord}). Let $\Lambda\in\mho\subset\Omega$ be a Delone multiset which contains the origin and $R'$ the radius for the relative density of $\Lambda$. As such, its associated scatterer configuration $\mathcal{S}_\Lambda$ not only covers the origin, but the origin is the center of mass of the scatterer which covers the origin. Suppose this is a scatterer of type $i$, so $\Lambda \in \mathcal{C}_i$. If $\Lambda'\in \mho$ is any other Delone multiset such that its scatterer centered at the origin is of type $i$ then it also belongs to $\mathcal{C}_i$. By finite local complexity, there are finitely many clusters in a ball of radius $4R'$ which appear in all the Delone multisets of $\Omega$. As such, if $\Lambda'\in\mathcal{C}_i\subset\mho$, then there are finitely configurations in which the scatterer at the origin can be surrounded by scatterers intersecting $B_{4R'}(0)$. This partitions $\mathcal{C}_i$ into finitely many subsets, one for each local configuration around the origin up to a distance of $4R'$, that is, depending on the scatterer configurations around a ball of radius $4R'$. Each of these elemens of the partition are further partitioned by the finitely many possible ways in which one finds scatterer configurations in balls of radius $8R'$ around the origin, and so on. This choice of finitely many options at every stage makes $\mathcal{C}_i$ a Cantor set.

If $\mathcal{S}$ is $\mathcal{T}$-equivariant, with $\mathcal{T}$ a hierarchical, aperiodic, repetitive tiling of finite local complexity, then $\partial \Sigma\subset \Omega_\mathcal{T}$ is obtained as follows. Let $\mathcal{S}_0\subset \Gamma_0$ the open set on the approximant that lifts to the collection of scatterers on $\mathbb{R}^d$ (Lemma \ref{lem:level0}) and let $\Sigma = \varpi^{-1}_0(\mathcal{S}_0)$ be the lift of the model scatterers to the tiling space. Then $\partial \Sigma = \varpi^{-1}_0(\partial \mathcal{S}_0) $. 
In addition, $\mathcal{T}\in\partial\Sigma$ if $\mathcal{T}$ is a tiling such that the origin lies on the boundary of a scatterer.

\subsection{The compact model}
\label{subsec:model}
As in the case of periodic scatterer configurations for the Lorentz gas, it will be easier to work with a factor system defined on a compact set. We describe the construction in this section. 

Let $\Omega$ be a tiling space (or pattern space) corresponding to an aperiodic, repetitive tiling (or Delone multiset) with finite local complexity, and let $\Sigma\subset \Omega$ be a subset which defines for any $p\in\Omega$ an aperiodic $p$-equivariant collection of $C^3$ convex scatterers $\mathcal{S}_p \subset \mathbb{R}^d$  of category A. 
Points on boundaries $\partial\mathcal{S}_p$ of scatterers will be given directions away from scatterers, and this is parametrized by the upper half of the $d-1$ unit sphere $S^{d-1}_+\subset T_p\mathbb{R}^d$.



For $(p,v_0)\in       ( \Omega\backslash \bar{\Sigma}     \times S^{d-1}  )      ~\cup  ~ ( \partial \Sigma  \times  S^{d-1}_+ ) $ define 
$$\tau_0 := \min \left\{r>0: \varphi_{rv_0}(p)   \in    \partial \Sigma     \right\}.$$
Let  $n_1\in S^{d-1}$ be the unit normal vector at the point $  \tau_0  v_0 \in\partial\mathcal{S}_p$ (equivalently, the unit normal vector to the scatterer at the origin in $\varphi_{\tau_0 v_0}(p)$).
We then define $v_1 := v_0- 2\langle v_0,n_1\rangle n_1$. 
 Recursively, define for $k\geq 1$
$$V_{k-1} : = \sum_{i=0}^{k-1}  \tau_i  v_i , \hspace{.35in}  \tau_k := \min \left\{r>0: \varphi_{rv_{k}} \circ \varphi_{V_{k-1}}       (p)   \in    \partial \Sigma     \right\},  $$
 $n_{k+1}\in S^{d-1}$ to be the unit normal vector at the point $V_k\in \partial\mathcal{S}_p $
 (equivalently, the tangentially normal at $\varphi_{\tau_k v_{k}} \circ \varphi_{V_{k-1}}    (p)\in \partial \Sigma),   $ and
$$  v_{k+1} = v_{k}-2\langle v_{k}, n_{k+1}\rangle n_{k+1}.   $$
Since our starting points were arbitrary, any $(p,v_0)\in         ( \Omega\backslash \bar{\Sigma}     \times S^{d-1}  )      ~\cup  ~ ( \partial \Sigma  \times  S^{d-1}_+ ) $  defines sequences $\{(\tau_i,v_i)\}$ with $|\tau_{i+1}-\tau_i|\leq M$ by finite horizon (see the beginning of Section \ref{sec:gases}).

For $k\in\mathbb{N}$ denote $L_{k} : = \sum_{i=0}^{k}  \tau_i $. 
For any pair $(p,v_0)\in         ( \Omega\backslash \bar{\Sigma}     \times S^{d-1}  )      ~\cup  ~ ( \partial \Sigma  \times  S^{d-1}_+ )$, 
the Lorentz gas trajectory for the scatterer configuration $\mathcal{S}_p$ 
starting from the origin and in direction $v_0$ is given for all $t\geq 0$ by

$$ \phi^{p}_t(\bar{0}, v_0) =   \left(   V_k  +   \left(     t-   L_k   \right)v_{k+1} ,       v_{k+1}       \right) =:   (  \phi^{p,v_0}_t(\bar{0}),  v_{k+1}   )  ,$$
where $k$ is such that $    L_k    \leq   t <  L_{k+1}    . $
More generally, for $\omega \in \mathbb{R}^d  \setminus   \mathcal{S}_p$ we define $\phi^{p}_t(\omega, v_0)= \phi^{\varphi_{\omega}(p)}_t(\bar{0}, v_0).$  
As a system evolving in the tiling space (or pattern space), the system is defined by 
$$\bar{\phi}_t(p, v_0) =  \left( \varphi_{ \phi^{p,v_0}_t(\bar{0}) }(p)   , v_{k+1}   \right)$$ 
for all $t\geq 0$. Note that this implies that $\bar{\phi}_{L_k}(p,v_0)\in \partial\Sigma$ for all $k\in\mathbb{N}$.

Since the collisions at $\partial \mathcal{S}_p$ serve as a Poincar\'e section of the flow 
$$\phi^p_t : ( \mathbb{R}^d \setminus  \overline{\mathcal{S}_p}    \times  S^{d-1}   )  \cup   (\partial \mathcal{S}_p   \times S^{d-1}_+)  \to    ( \mathbb{R}^d \setminus  \overline{\mathcal{S}_p}    \times  S^{d-1}   )  \cup   (\partial \mathcal{S}_p   \times S^{d-1}_+), $$ 
the set $\partial \Sigma\times S^{d-1}_+$ serves as a Poincar\'e section of the flow
$$ \bar{\phi}_t  :   ( \Omega\backslash \bar{\Sigma}     \times S^{d-1}  )      ~\cup  ~ ( \partial \Sigma  \times  S^{d-1}_+ )   \to   ( \Omega\backslash \bar{\Sigma}     \times S^{d-1}  )      ~\cup  ~ ( \partial \Sigma  \times  S^{d-1}_+ ).   $$
Denote by $f:\mathcal{D}\rightarrow \mathcal{D}$ the Poincar\'e map, with $\mathcal{D} := \partial \Sigma\times S^{d-1}_+$ called the \textbf{collision space}. Using (\ref{eqn:localCoord}), the collision space inherits a local product structure from $\Omega$ as follows. A neighborhood in $\partial \Sigma \times S^{d-1}_+$ is defined by
\begin{enumerate}
\item a neighborhood in $\partial S_i$ for some $i\in\overline{\mathfrak{C}(\mathcal{S})}$ describing the translation class of the scatterer where the collision is taking place (a small subset homeomorphic to a disk of dimension $d-1$);
\item a clopen set in $\mathcal{C}_i$ describing a local configuration of scatterers around the a collision point, that is, a cylinder set defined by a cluster $\mathcal{P}$ of the Delone multiset $\Lambda\in\partial\Sigma$ with $\bar{0}\in\mathcal{P}\subset \Lambda$;
\item and a neighborhood in $S^{d-1}_+$ homeomorphic to a disc of dimension $d-1$ describing the direction in which the trajectory goes after the collision.
\end{enumerate}
Thus the collision set can be given coordinates from the set
\begin{equation}
  \label{eqn:coordinates}
  \partial \Sigma \times S^{d-1}_+= \left(\bigsqcup_i \partial S_i\times \mathcal{C}_i\right)\times S^{d-1}_+.
\end{equation}
In other words, any point in $\mathcal{D}$ can be written in local coordinates as $(r,c,v)$ where $r\in \partial S_i$, $c \in \mathcal{C}_i$ belongs to a transverse cylinder set of the pattern space $\Omega$, and $v\in S^{d-1}_+$ is a direction. We note that each path connected component of this space is isometric to $\partial S_i\times S^{d-1}_+$ for some $i$. We will refer to these components as the \textbf{connected components} of $\mathcal{D}$. For $x\in \mathcal{D}$ we will denote the connected component containing $x$ by $\mathcal{L}_x$. A set $A\subset \mathcal{D}$ is \textbf{leafwise open} if $A\subset \mathcal{L}$ for some connected component $\mathcal{L}\subset \mathcal{D}$ and open in the subspace topology on $\mathcal{L}$ inherited from the topology of $\mathcal{D}$. A \textbf{leafwise open} neighborhood of $x$ is the intersection of an open neighborhood of $x$ in $\mathcal{D}$ and the connected component containing $x$.



\begin{definition}
  A \textbf{canonical cylinder set} of the collision space is a subset $C_{I,\mathcal{P},D}\subset \partial \Sigma\times S^{d-1}_+$ of the collision space expressed in the local coordinates (\ref{eqn:coordinates}) where $I\subset \partial S_i$ is a small ball, $\mathcal{P}$ is a patch/cluster around the origin having a point of type $i$ at the origin, and a small ball of direction vectors in $D\subset S^{d-1}_+$.
\end{definition}
\begin{remark}
  \label{rem:canonical}
  Some remarks:
  \begin{enumerate}
  \item Since patches define subsets of $\mho$, a canonical cylinder set is partly defined by a subset $C\subset \mho$.
  \item The topology of the collision space $\partial \Sigma\times S^{d-1}_+$ is generated by canonical cylinder sets. Moreover, by compactness, for any $\varepsilon>0$ there exist finitely many open balls $\mathfrak{I}_\varepsilon = \{I_i\}$, $I_i\subset \mathcal{S}_0$ of radius $\varepsilon$, patches $\mathfrak{P}_\varepsilon = \{\mathcal{P}_j\}$ of radius $2\varepsilon^{-1}$, and finitely many open balls $\mathfrak{D}_\varepsilon = \{D_k\}$ of $S^{d-1}_+$ of radius $\varepsilon$ such that 
$\{C_{I,\mathcal{P},D}  ~ |     ~    I\in \mathfrak{I}_\varepsilon, ~ \mathcal{P}\in \mathfrak{P}_\varepsilon, ~ D\in \mathfrak{D}_\varepsilon       \}$ 
covers the collision space $\partial \Sigma\times S^{d-1}_+$. 
  \end{enumerate}
\end{remark}

The pattern space $\Omega$ is a foliated space and, as such, so is the set $\Omega\times S^{d-1}_+$ by taking the leaves to be of the form $\mathcal{L}\times S^{d-1}_+$ for any leaf $\mathcal{L}$ in the foliation of $\Omega$. As such, the collision space $\mathcal{D}$ inherits a foliated structure, where the leaves are the connected components in (\ref{eqn:coordinates}).
This allows us to define (leafwise) tangential Jacobians as follows: for $x = (r,c,v)\in \mathcal{D}$ which does not lie on a discontinuity set of $f$, the tangential Jacobian $D_{x}^tf$ as a map from a leafwise open neighborhood of $x$ to a leafwise open neighborhood of $f(x)$.

 \subsection{Planar aperiodic Lorentz gases}

  We now focus on the planar case, when $d=2$. In this case, the prefered coordinates in (\ref{eqn:coordinates}) are parametrized by $(\rho,c,\theta)$, where $\rho\in \partial S_i$ is an arclength parameter, $c\in \mathcal{C}_i$ is an element of the zero dimensional transverse set, and $\theta \in[-\pi/2,\pi/2]$ is the angle of collision at $\rho$, chosen so that the unit outward normal corresponds to $\theta=0$. 
We will denote the corresponding projection maps $ \pi_i$ from (\ref{eqn:coordinates}), noting that the last coordinate takes values in $[-\pi/2,\pi/2]$. 
Denote by $\Xi_1$ the \textbf{set of singularities of the map $f$}. These correspond to tangential collisions: $\bar{p}\in \Xi_1$ if and only if $\pi_3(f(\bar{p})) = \pm\pi/2$. These are ``Cantorized'' unions of compact smooth curves. In other words, there is a finite partition of the Cantor sets $\mathcal{C}_i$ in (\ref{eqn:coordinates})
  $$\mathcal{D} = \left( \bigsqcup_{i} \partial S_i\times \bigsqcup_{j=1}^{s(i,1)}\mathcal{C}_{i,j}^1 \right)\times [-\pi/2,\pi/2]$$
  such that there are finitely many curves $\gamma_{i,j}\subset \partial S_i\times [-\pi/2,\pi/2]$ with $\Xi_1$ homeomorphic to
  \begin{equation}
    \label{eqn:cantorized}
    \bigsqcup_{i,j} \gamma_{i,j}\times \mathcal{C}_{i,j}^1.
  \end{equation}
 Likewise denote by $\Xi_{-1}$ the set of singularities of $f^{-1}$ and recursively,
  \begin{equation}
    \label{eqn:disc}
    \Xi_{\pm(n+1)} = \Xi_{\pm n}\cup f^{\mp 1}(\Xi_{\pm n}),  \hspace{.25in} \Xi^\pm = \bigcup_{n\geq 0}\Xi_{\pm n},  \hspace{.25in} \mbox{ and } \hspace{.25in}\Xi = \Xi^+\cup \Xi ^-.
  \end{equation}
Likewise for each $n\in\mathbb{N}$ there is a finite partition $\{\mathcal{C}_{i,j}^n\}$ of the union of the Cantor sets $\{\mathcal{C}_i\}$ such that $\Xi_{\pm n}^\pm$ can be expressed as the cantorized union of curves as in (\ref{eqn:cantorized}).
  
  Let $\mu$ be a Borel probability measure on the tiling (or pattern) space $\Omega$ which is ergodic with respect to the $\mathbb{R}^2$ action and recall its local product structure $\mathrm{Leb}\times \nu$, where $\nu$ is the measure on the zero dimensional transversal defined by frequencies of patches (or clusters). Since the boundaries of the scatterers are sufficiently smooth ($C^3$), there is a naturally induced measure on $\partial S_i\times \mathcal{C}_i \times [-\pi/2,\pi/2]$ which has the local product structure of $\mathrm{Leb}\times\nu\times\mathrm{Leb}$. The sets of singularities $\Xi^\pm$ have zero $\mathrm{Leb}\times\nu\times\mathrm{Leb}$ measure. Consider the measure $\bar{\mu}$ on $\mathcal{D}$ expressed in local coordinates as $\bar{\mu}:= \cos\theta \, d\rho\, d\nu \, d\theta$. The extended set of collisions with eventual singularities $\Xi$ has zero $\bar{\mu}$ measure.
\begin{proposition}\label{prop:invmeasure}
    The map $f:\partial\Sigma\times [-\pi/2,\pi/2]\longrightarrow \partial\Sigma\times [-\pi/2,\pi/2]$ preserves the measure $\bar{\mu}$.
\end{proposition}
\begin{proof}

  Let $\pi_i$, $i=1,2,3$, be the different projections from the coordinates in (\ref{eqn:coordinates}). Let $A$ be a canonical cylinder set of the form $A = U\times \mathcal{C}_i\times I$ with the property that $A$ does not intersect the set of discontinuities of $f$. The set $A$ parametrizes an open set of trajectories starting at a scatterer of type $i$ which travel together before hitting another scatterer. By finite local complexity, there are finitely many families of trajectories defined by $A$. Each of these different families of trajectories is 
determined by a two-point cluster $\mathcal{P}_{A,j}$, where one point is of type $i$, being the center of mass of a scatterer of type $i$ from which the trajectory is emanating and located at $\bar{0}$, and where the other point in the cluster is determined by the center of mass and type of the scatterer hit in the first collision. We index all possible families of trajectories by $\mathfrak{I}(A)$ and stress that this is a finite set by finite local complexity.
We then have that
  $$A = \bigsqcup_{j\in\mathfrak{I}(A)} A_j \text{ with }\sum_{j\in\mathfrak{I}(A)}\mathrm{freq}(\mathcal{P}_{A,j}) = \nu(\mathcal{C}_i).$$

We can now show the invariance of the measure:
\begin{equation}
  \label{eqn:measure}
    \begin{split}
      \bar{\mu}(f(A)) &= \sum_{j\in\mathfrak{I}(A)} \bar{\mu}(f(A_j)) = \sum_{j\in\mathfrak{I}(A_j)} \iiint_{f(A_j)} \cos\theta\, d\theta\, d\nu\, d\rho \\
      &= \sum_{j\in\mathfrak{I}(A_j)} \mathrm{freq}(\mathcal{P}_{A,j})\int_{\pi_1(f(A_j))}\int_{\pi_3(f(A_j))} \cos\theta\, d\theta\, d\rho \\
      &= \sum_{j\in\mathfrak{I}(A_j)} \mathrm{freq}(\mathcal{P}_{A,j})\int_{\pi_1(A_j)}\int_{\pi_3(A_j)} \cos\theta\, d\theta\, d\rho \\
      &= \sum_{j\in\mathfrak{I}(A_j)} \iiint_{A_j} \cos\theta\, d\theta\, d\nu\, d\rho = \int_A\,d\bar{\mu} = \bar{\mu}(A),
    \end{split}
  \end{equation}
  where we have used the preservation of the leafwise measure \cite[Lemma 2.35]{CM:book} to obtain the fourth equality. For smaller canonical cylinder sets $A$ (i.e. ones where we use a much smaller set $C\subset \mathcal{C}_i$) we do not actually need to break up the set $A$ by clusters and a simpler version of (\ref{eqn:measure}) will hold, showing the invariance of the measure $\bar{\mu}$.
\end{proof}
\section{Stable and unstable manifolds}
\label{sec:womanifolds}
In this section we collect results about the properties of stable and unstable manifolds. Most of the results follow from the theory of Sinai billiards, so we frequently refer to a standard reference for these objects \cite{CM:book}. The reason that most results follow from the theory of Sinai billiards can be briefly summarized as follows. We consider systems with finite local complexity, therefore having finitely many scatterer types, which results in uniform (upper and lower) bounds on their curvature.
Another fact that follows from finite local complexity is finitely many possible local configurations of scatterers, which gives uniform lower bound on the time between collisions. 
In addition, we assume that our system has finite horizon, giving an upper bound on the time between collisions. Those facts lead to an existence of stable and unstable cones on the (leafwise) tangent bundle (see Definition \ref{def:leafwisetangentcolspace}) of the 
collision space $\mathcal{D}= \partial\Sigma \times [-\pi/2,\pi/2]$.
As in the case of periodic Sinai Billiards, those cones are strictly invariant and exhibit uniform contraction and expansion respectively.
This implies that 
leaves of the foliation of $\mathcal{D}$ given by the orbits of the $\mathbb{R}^2$-action admit local stable and unstable manifolds at $\mbox{Leb}$-almost every point.
In addition, leafwise analysis of stable/unstable curves, stable/unstable manifolds, stable/unstable homogeneus manifolds, as well as leafwise analysis of singularity curves is identical to the periodic case.
Consequently, all properties of periodic billiards established in \cite[Chapter 4 and Chapter 5]{CM:book} have their leafwise versions in our setting.
This includes results regarding lengths and regularity of stable/unstable (homogeneus) manifolds, results concerning the regularity of the holonomy map, growth lemmas, and the Fundamental theorem of the theory of dispersing billiards.
In this section we summarize some of the main facts, skipping many of them. We emphasize that arguments in our paper rely on two major results: the absolute continuity of the holonomy map, as well as the Fundamental theorem of the theory of dispersing billiards.

Let $\mathcal{T}$ be a repetitive tiling of $\mathbb{R}^2$ of finite local complexity. Consider a $\mathcal{T}$-equivariant collection of $C^3$ dispersing scatterers (that is scatterers of positive curvature) of category A. 

\begin{definition}\label{def:leafwisetangentcolspace}
Let $T\partial\Sigma$ denote the sub-bundle of the (leafwise) tangent bundle $T\Omega_\mathcal{T}$ corresponding to the restriction of $T\Omega_\mathcal{T}$ to $\partial\Sigma$.
We define the \textbf{ (leafwise) tangent bundle of the 
collision space} $\mathcal{D}= \partial\Sigma \times [-\pi/2,\pi/2]$
to be $$ T\mathcal{D}  :=   T\partial\Sigma  \times [-\pi/2,\pi/2]. $$

\end{definition}

Define the \textbf{stable/unstable cones at $p\in\mathcal{D}$} to be
$$\mathcal{C}_p^u = \{(dr,d \phi)\in T_p\mathcal{D} :\mathcal{K}_p\leq d\phi/dr\leq \mathcal{K}_p + \cos\phi /\tau_{-1}  \},$$
where $\mathcal{K}_p$ is the curvature of the scatterer at $p$, and $\tau_{-1}$ is the time between the colision at $f^{-1}(p)$ and $p$, and
$$\mathcal{C}_p^s = \{(dr,d \phi)\in T_p\mathcal{D} : - \mathcal{K}_p - \cos\phi /\tau_{1} \leq d\phi/dr\leq - \mathcal{K}_p  \}.$$
Both families of cones are strictly invariant \cite[\S 4.4]{CM:book}:
\begin{equation}
  \label{eqn:coneInv}
  D^t_pf(\mathcal{C}_p^u)\subset \mathrm{int}\, \mathcal{C}_{f(p)}^u\cap \{0\}\hspace{.6in}\mbox{ and }\hspace{.6in} D^t_pf^{-1}(\mathcal{C}_p^s)\subset \mathrm{int}\, \mathcal{C}_{f^{-1}(p)}^s\cap \{0\}.
\end{equation}

In addition, setting $\Lambda:= 1 + 2\tau_{min}\mathcal{K}_{\min}>1,$ where $\tau_{min}$ denotes the shortest time between collisions and $\mathcal{K}_{\min}$ is the lower bound on the curvature of a scatterer,
we obtain that \cite[\S 4.4]{CM:book}

$$ || D^t_pf^n(dx^u) ||  \geq  \hat{c}\Lambda^n ||dx^u||    \hspace{.6in}\mbox{ and }\hspace{.6in}    || D^t_pf^{-n}(dx^s) ||  \geq  \hat{c}\Lambda^n ||dx^s||      $$
for any $p\in\mathcal{D}$, some $\hat{c}>0$, any $dx^u=(dr^u,d\phi^u)\in \mathcal{C}_p^u$ and any $dx^s=(dr^s,d\phi^s)\in \mathcal{C}_p^s$
thus establishing uniform leafwise hyperbolicity of the map $f:\mathcal{D}\rightarrow \mathcal{D}$. 
However, the transverse direction does not exhibit any hyperbolic behavior. Indeed, if one considers a periodic point $p = (r,c,\phi)$ it is immediate that any point $p' = (r,c',\phi)$ for $c'\in\mathcal{C}_0$ close enough to $c\in\mathcal{C}_0$ will also be periodic. In fact, the same argument shows that any \emph{bounded} orbit of the aperiodic Lorenz gas comes in a Cantorized family of bounded orbits. The following Lemma makes precise how the compact model $f$ is a partially hyperbolic system for any orbit.
First, let $B_\varepsilon^\mho(x)$ be the ball of radius $\varepsilon$ containing $x$ along the transversal direction. That is, for $x = (r,c,\phi)$
$$B^\mho_\varepsilon(x) = \{z = (r',c',\phi')\in \mathcal{D}: r'=r,\, \phi'=\phi,\; d(x,z) <\varepsilon \}.$$
Note that this defines a patch $\mathcal{P}$ and tile $t$ such that $B_\varepsilon^\mho(x) = \mathcal{C}_{\mathcal{P},t}$.
\begin{lemma}
  \label{lem:partial}
  There exists a $\tau>0$ such that for all $\varepsilon\in (0,\tau)$,
  $$f^k\left(B^\mho_{\varepsilon}(x)\right)\subset B^\mho_{2\varepsilon}\left( f^k(x) \right)$$
  for all $0\leq k < \lfloor \tau/\varepsilon \rfloor /2$.
\end{lemma}
\begin{proof}
  By finite horizon there is a $M>0$ such that the time/distance between two collissions are bounded above by $M$. Likewise, since the scatterers are $\mathcal{T}$-equivariant, there is an $R$ so that if two tilings agree on a ball of radius $R$ then their scatterers around the origin are translation-equivalent. Let $\varepsilon<\tau:= 1/4R$.

  If $z\in B_{\varepsilon}^\mho(x)$, then the patches around the corresponding tilings contained in a ball of radius $\varepsilon^{-1}$ are identical, and so the first $\lfloor  (\varepsilon^{-1}-2R)/M \rfloor$ collisions will be identical. That is, the path followed by the billiard trajectory after $\lfloor  (\varepsilon^{-1}-2R)/M \rfloor$ collisions starting at the initial conditions defined by $x$ will be indistinguishable from those defined by $z$. Thus, after $k<(\lfloor  (\varepsilon^{-1}-2R)/M )\rfloor/2$ collisions, the patches of radius $\varepsilon^{-1}/2$ around the origin of $f^k(x)$ and $f^k(z)$ will be the same, meaning that $f^k(z)\in B_{2\varepsilon}(f^k(x))$, proving the claim. Note that we have proved that for all $k$ large enough, $f^i\left(B^\mho_{\frac{1}{kM-2R}}(x)\right)\subset B^\mho_{\frac{2}{kM-2R}}\left(f^i(x)\right)$ for all $0\leq i < k/2$.
\end{proof}

By the local structure of $\mathcal{D}$, any curve $W\subset \mathcal{D}$ of finite length is necessarily contained path-connected components of $\mathcal{D}$, that is, $W$ is contained in a leaf of the foliated space $\mathcal{D}$. A curve $W$ is $C^m$ if it is leafwise-$C^m$. A $C^m$ curve $W$ is \textbf{unstable} if for any $p\in W$ the leafwise tangent line to $W$ at $p$ is contained in $\mathcal{C}^u_p$. A $C^m$ curve $W$ is \textbf{stable} if for any $p\in W$ the leafwise tangent line to $W$ at $p$ is contained in $\mathcal{C}^s_p$.

\begin{definition}
  A smooth curve $W\subset \mathcal{D}$ is an \textbf{unstable manifold} if $f^n$ is smooth on $W$ for all $n<0$ and $|f^{-n}(W)|\rightarrow 0$ as $n\rightarrow \infty$. A smooth curve $W\subset \mathcal{D}$ is a \textbf{stable manifold} if $f^n$ is smooth on $W$ for all $n>0$ and $|f^{n}(W)|\rightarrow 0$ as $n\rightarrow \infty$.
\end{definition}

The sets $\Xi^+$ and $\Xi^-$ of discontinuities defined in (\ref{eqn:disc}) are dense in $\mathcal{D}$; the proof from the periodic case extends to the setting here \cite[Lemma 4.55]{CM:book}. The set $\Xi^+\subset \mathcal{D}$ is the union of smooth unstable curves while $\Xi^-$ is the union of smooth stable curves. For each $n\geq 0$, the set $\Xi_n$ consists of a family of transversally locally constant curves. More precisely, consider the set $\Xi_1$, the set of discontinuities of the map $f$, and a point $p\in \Xi_1$. Denote by $\Xi^p_1$ the path connected component of $\Xi_1$ containing $p$. The point $p$ corresponds to a type of billiard trajectory which leaves a scatterer and has a tangential collision with another scatterer in time less than $M$. If the tiling is repetitive, then the patch allowing a translation-equivalent type of tangential collision is found within a bounded distance of every point. Each of these patches correspond to grazing collisions of the same type, and they correspond to other connected components of $\Xi_1$ which are all contained in a transverse neighborhood of $\Xi_1^p$. More generally, for $n>1$, the set $\Xi_n$ the union of finitely many transversally-locally-constant curves. As $|n|$ increases, the transverse neighborhood containing path connected components of $\Xi_n\backslash \Xi_{n-1}$ decreases in size. Another way of stating this is saying that if $p = (r,c,\phi)\in \Xi_n$ then $p'= (r,c',\phi)\in \Xi_n$ for all $c'$ close enough to $c$ in $\mathcal{C}_0$.

For $k\geq k_0$ (some large constant) and a transverse coordinate $c\in\mathcal{C}_0$, define the \textbf{homogeneity strips at $c$} to be
\begin{equation}
  \label{eqn:strips}
  \begin{split}
    \mathbb{H}_k^c := \left\{(r,c,\phi): \frac{\pi}{2}-k^{-2}<\phi<\frac{\pi}{2}-(k+1)^{-2}\right\} \hspace{.5in}&\mbox{ with }\hspace{.25in}\mathbb{H}_k := \bigcup_{c\in\mathcal{C}_0} \mathbb{H}_k^c, \hspace{.1in}\mbox{ and }\\
    \mathbb{H}_{-k}^c := \left\{(r,c,\phi): -\frac{\pi}{2}+(k+1)^{-2}<\phi<-\frac{\pi}{2}+k^{-2}\right\}\hspace{.25in}&\mbox{ with }\hspace{.25in}\mathbb{H}_{-k} := \bigcup_{c\in\mathcal{C}_0} \mathbb{H}_{-k}^c.
  \end{split}
\end{equation}
The boundaries of the homogeneity strips are denoted by
$$\mathbb{S}_k = \{(r,c,\phi): |\phi| = \pi/2 - k^{-2}\}$$
for $|k|\geq k_0 $ with union
$$\mathbb{S} = \bigcup_{|k|\geq k_0} \mathbb{S}_k.$$
\begin{definition}
A stable or unstable curve $W\subset \mathcal{D}$ is \textbf{weakly homogeneous} if $W$ is contained in one strip $\mathbb{H}_k$. An unstable manifold $W$ is \textbf{homogeneous} if $f^{-n}(W)$ is weakly homogeneous for all $n\geq 0$. A stable manifold $W$ is homogeneous if $f^n(W)$ is weakly homogeneous for all $n\geq 0$.
\end{definition}
If $W$ is a weakly homogeneous stable or unstable curve contained in $\mathbb{H}_k$ and contains the point $x = (r,c,\phi)$ then the length can be bounded by (see \cite[Lemma 5.10]{CM:book})
\begin{equation}
  \label{eqn:weakHom}
|W|\leq C (|k|+1)^{-3} \leq C \cos^{3/2}(\phi).
\end{equation}
As is standard in the theory of planar dispersive billiards, homogeneous manifolds will allow us to keep control of distortion estimates. To this end, following \cite[5.4]{CM:book}, we introduced introduce a new collision space adapted to the homogeneity strips.

  For each $k\in\mathbb{Z}$ with $|k|\geq k_0$ and using (\ref{eqn:coordinates}), let
  $$\mathcal{D}_{i,k}:= \left( \partial S_i\times\mathcal{C}_i\times [-\pi/2,\pi/2]\right) \cap \mathbb{H}_k.$$
  The \textbf{new collision space} $\mathcal{D}_\mathbb{H}$ is now defined to be the disjoint union of the $\mathcal{D}_{i,k}$. As such, we have that $\partial \mathcal{D}_{\mathbb{H}} = \mathbb{S}$. The map $f$ from the compact model naturally acts on $\mathcal{D}_\mathbb{H}$ and this restriction is denoted by $f_\mathbb{H}$, but note that it is not defined on $f^{-1}(\mathcal{S}_0)$ since $\mathcal{S}_0$ is not part of the new collision space. In the new collision space we extend the set of discontinuities (\ref{eqn:disc}) by defining for all $n\in \mathbb{N}$
  $$\Xi^\mathbb{H}_n:= \Xi_n \cup\left( \bigcup_{m=0}^n f^{-m}(\mathbb{S}) \right), \hspace{.125in} \Xi^\mathbb{H}_{-n}:= \Xi_{-n} \cup\left( \bigcup_{m=0}^n f^{m}(\mathbb{S}) \right) \hspace{.125in}\mbox{ and } \hspace{.125in} \Xi_{\pm \infty}^{\mathbb{H}} = \bigcup_{n\geq 0}\Xi_{\pm n}^\mathbb{H}.$$
  
Let $x\in \mathcal{D}\backslash \Xi_{\pm n}^\mathbb{H}$ and define $\mathcal{Q}_{\pm n}^t(x)$ to be the path connected component of the set $\mathcal{D}\backslash \Xi_{\pm n}^\mathbb{H}$ which contains $x$ (the superscript $t$ is meant to denote tangential), and $\mathcal{Q}_{\mp n}(x)$ the union of all isometric connected components transversally close to $\mathcal{Q}^t_{\mp n}(x)$. In other words, if $x = (r,c,\phi)\in\mathcal{D}\backslash \Xi^\mathbb{H}_{\pm n}$, then there exists a transverse neighborhood $U_x^{\pm n}\subset \bigcup_i \mathcal{C}_i$ of $c$ such that
$$\mathcal{Q}_{\pm n}(x) = \bigcup_{c'\in U_x^{\pm n}} \mathcal{Q}^t_{\pm n}(r,c',\phi).$$
Note that by compactness of $\mathcal{D}$ the set $\mathcal{D}\backslash \Xi_{\pm n}^\mathbb{H}$ is the union of finitely many sets of the form $\mathcal{Q}_{\pm n}(x)$. That is, for any $n\neq 0$ there exist finitely many $p_1,\dots, p_{k_n}$ such that
\begin{equation}
  \label{eqn:finitelyConnected}
  \mathcal{D}\backslash \Xi_{\pm n}^{\mathbb{H}}=\bigcup_{i=1}^{k_n}\mathcal{Q}_{\pm n}(p_i).
\end{equation}
For a point $c$ in the ultrametric space $\bigcup_i \mathcal{C}_i$, let $\mathcal{B}_n(c)$ be the ball of diameter $2^{-n}$ containing $c$.

For any $n\in\mathbb{N}$ we can further give a countable collection of sets
\begin{equation}
  \label{eqn:Qsets}
  \mathcal{Q}_{-n,m}(r,c,\phi) := \mathcal{Q}_{-n}(x)\cap( \partial \mathcal{S}_0\times \mathcal{B}_m(c)\times[-\pi/2,\pi/2]).
\end{equation}
Sets $\mathcal{Q}_{n,m}(x)$ are similarly defined for $n,m\in\mathbb{N}$. As such, we have that
$$\mathcal{Q}_{-n}^t(x) = \bigcap_{m>0}\mathcal{Q}_{-n,m}(x),\hspace{.75in}\mbox{ and }\hspace{.75in}\mathcal{Q}_{n}^t(x) = \bigcap_{m>0}\mathcal{Q}_{n,m}(x).$$
Define for $x\in\mathcal{D}\backslash \Xi$ and $m\in\mathbb{N}$:
\begin{equation}
\label{eqn:Wcu}
  \begin{split}
    \bar{W}^{cu}_m(x)&:= \bigcap_{n>0}\overline{\mathcal{Q}_{-n,m}(x)}, \hspace{.85in}\hspace{.95in} \bar{W}^{u}(x):= \bigcap_{n>0}\overline{\bigcap_{m>0}\mathcal{Q}_{-n,m}(x)},\\
    \bar{W}^{cs}_m(x)&:= \bigcap_{n>0}\overline{\mathcal{Q}_{n,m}(x)}, \hspace{.85in} \mbox{ and }\hspace{.75in}  \bar{W}^{s}(x):= \bigcap_{n>0}\overline{\bigcap_{m>0}\mathcal{Q}_{n,m}(x)}
  \end{split}
\end{equation}
and let $W^s(x)$ and $W^u(x)$, respectively, be the curves $\bar{W}^s(x)$ and $\bar{W}^u(x)$ without their endpoints. They are the maximal homogeneous stable and unstable manifolds going through $x$.
\begin{lemma}
For $x = (r,c,\phi)\in \mathcal{D}\backslash \Xi$, we have that
  \begin{equation}
    \label{eqn:Wu}
  W^u(x)\subset \bigcap_{n>0} \mathcal{Q}^t_{-n}(r,c,\phi) \hspace{.55in} \mbox{ and }\hspace{.55in}   W^s(x)\subset \bigcap_{n>0} \mathcal{Q}^t_{n}(r,c,\phi).
\end{equation}
  Moreover if the boundary $\partial \mathcal{S}_0$ of the scatterers are $C^\ell$ curves for some $\ell\geq 3$, then the curves $W^u(x)$ and $W^s(x)$ are $C^{\ell-2}$ smooth with its derivatives uniformly bounded and its $(\ell-2)^{nd}$ derivative Lipschitz with uniform constant.
\end{lemma}
\begin{proof}
  \cite[\S 4.11]{CM:book}.
\end{proof}


Let $W$ be a smooth homogeneous unstable manifold and $p\in W$. The point $p$ divides $W$ into two subsegments, and denote by $r^u(p) = r_W(p)$ the length (in the induced Euclidean metric from the flat metric on the leaves of the tiling space $\Omega$) of the shorter one. If $W^u(x) = \varnothing$, then we set $r(x) = 0$. The proof for the periodic Lorenz gas with finite horizon carries over with very minor modifications to the proof of the statement \cite[\S 5.5]{CM:book}.
\begin{theorem}
  $W^u(x)$ exists (i.e. $r^u(x)>0$) for almost every $x\in\mathcal{D}$. Moreover
  \begin{equation}
    \label{eqn:linearTail}
    \bar{\mu}(\{x\in \mathcal{D}: r^u(x)<\varepsilon\})< C\varepsilon
  \end{equation}
  for some constant $C$ independent of $x$ and for all $\varepsilon>0$.
\end{theorem}
\subsection{Unstable densities and stable holonomies}
The unstable partition $\{\xi^{u}\}$ of $\mathcal{D}$ consists of maximal homogeneous unstable manifolds of points. More precisely, the element of the partition containing a point $x$ is $\xi^{u}(x) = W^{u}(x)$ if $W^{u}(x) \neq \varnothing$, and otherwise $\xi^{u}(x) = \{x\}$. Likewise the stable partition $\{\xi^s\}$ is defined as $\xi^{s}(x) = W^{s}(x)$ if $W^{s}(x) \neq \varnothing$, and otherwise $\xi^{s}(x) = \{x\}$. As in the periodic case, the sets $\mathcal{Q}_{n,m}(x)$ in (\ref{eqn:Qsets}) form a countable generator for the stable and unstable partitions. As such, the stable and unstable partitions $\xi^s, \xi^{u}$ are measurable.

Given that the stable and unstable partitions are measurable, it follows that for $\mu$-almost every $x$ there are conditional measures 
induced by the measure $\bar{\mu}$ on the stable and unstable manifolds $W = W^{s/u}(x)$ denoted respectively by $\nu_W$.
As stated in Theorem 5.2 in \cite{CM:book}, those conditional measures are absolutely continuous with respect to the leaf volume on $W$ with $C^{\ell-2}$ density $\rho_W$. We record here important distortion estimates which follow from the analysis in \cite[\S 5.6]{CM:book}.
\begin{proposition}
  For any homogeneous unstable manifold $W$, for every $l = 1,\dots, \ell-2$ there is a $C_l>0$ such that
  $$\left| \frac{d^l}{dx^l} \ln \rho_W(x)\right|\leq \frac{C_l}{|W|^{2l/3}}.$$
  Moreover, for any $x,y\in W$,
  $$C_1^{-1}\leq e^{-C_2|W|^{1/3}}\leq \frac{\rho_W(x)}{\rho_W(y)}\leq e^{C_2|W|^{1/3}}\leq C_1$$
  for some $C_1,C_2>0$ which are independent of $W,x,y$.
\end{proposition}

\begin{definition}
  Let $W^1, W^2$ be two unstable curves. Denote by
  $$W^i_*:= \{x\in W^i:W^s(x)\cap W^{3-i}\neq \varnothing \}$$
  for $i=1,2$. The map $\mathbf{h}:W^1_*\rightarrow W^2_*$ taking the point $x\in W^1_*$ to $\mathbf{h}(x) = W^s(x)\cap W^2$ is called the \textbf{holonomy map}.\footnote{In \cite[Definition 5.41]{CM:book} the map $h$ is called the modified holonomy map, since it is obtained by sliding along homogeneous stable manifolds, hence it is a restriction of the usual holonomy map.}
\end{definition}
Note that for the holonomy map to be defined, it is necessary that the two unstable curves lie on the same leaf of the foliated space $\Omega$. Suppose the holonomy is defined at $x\in W^1$ with Lebesgue measure $m_1$ and mapped to $W^2$ with Lebesgue measure $m_2$. Define the \textbf{Jacobian of the holonomy map} to be
$$J\mathbf{h}(x) = \frac{d\mathbf{h}^{-1}_*m_2}{dm_1}(x).$$
\begin{theorem}
  There exists $C>1$ which depends only on the scatterers such that
  $$C^{-1}\leq j\mathbf{h}(x)\leq C$$
  for all $x\in W^1_*$. In addition, there is an $A>1$ which also only depends on the model scatterers such that
  $$A^{-\gamma-\delta^{1/3}}\leq J\mathbf{h}(x)\leq A^{\gamma+\delta^{1/3}},$$
  where $\delta$ is the distance between $x$ and $\mathbf{h}(x)$ and $\gamma$ the angle between the tangent vectors to the curves $W^1$ and $W^2$ at the points $x$ and $\mathbf{h}(x)$, respectively. 
\end{theorem}
The proof follows the proof of the periodic case, \cite[Theorem 5.42]{CM:book}.


The next result establishes that the holonomy map is leafwise dynamically H\"older continuous, that is H\"older continuous with respect to the separation time defined on the leaves of foliation of $\mathcal{D}$ given by the orbits of the $\mathbb{R}^2$-action.  
For two points $x,y\in\mathcal{D}$ on the same leaf define

$$ s_+(x,y):= \min\{  n\geq 0 ~ | ~   y\not\in \mathcal{Q}^t_n(x)  \}.   $$

The following can be proven using identical arguments as Proposition 5.48 in \cite{CM:book}.
\begin{proposition}
  There are constants $C>0$ and $\theta\in (0,1)$ such that
  $$|\ln J\mathbf{h}(x) - \ln J\mathbf{h}(y)| \leq C\theta^{s_+(x,y)}.$$
\end{proposition}
Let $W$ be a weakly homogeneous curve and $m_W$ the Lebesgue measure on it. For any $x\in W$ and $n\geq 0$ let $W_n(x)$ be the H-component of $f^n(W)$ containing $f^n(x)$ and by $r_n(x) = r_{W_n(x)}(f^n(x))$ the distance from $f^n(x)$ to the nearest point of $\partial W_n(x)$.
We now state the two growth lemmas, the proofs of which follow from the proof for the periodic case \cite[\S 5.10]{CM:book}.
\begin{theorem}
  There exist constants $\hat{\Lambda}>1$, $\theta_1\in(0,1)$, $c_1,c_2>0$ such that for any $n\geq 0$ and $\varepsilon>0$,
  \begin{equation}
    \label{eqn:firstGrowth}
    m_W(r_n(x)<\varepsilon) \leq c_1(\theta_1 \hat{\Lambda})^n m_W(r_0(x)<\varepsilon/\hat{\Lambda}^n) + c_2\varepsilon m_W(W).
  \end{equation}
  In addition there are constants $\zeta>0$ and $c_3>0$ such that for all $n\geq \zeta \left| \ln |W| \right|$ and $\varepsilon>0$,
  \begin{equation}
    \label{eqn:secondGrowth}
    m_W(r_n(x)<\varepsilon)\leq c_3\varepsilon m_W(W).
  \end{equation}
\end{theorem}

Finally, we state the analog of Sinai's fundamental theorem for dispersing billiards \cite[Theorem 5.70]{CM:book}.

\begin{theorem}
  \label{thm:fundamental}
  Let $x\in \mathcal{D}\backslash \Xi_\infty^{\mathbb{H}}$. For any $q>0$ and $A>0$ there exists an leafwise open neighborhood $\mathcal{U}^u_x\subset \mathcal{D}$ of $x$ such that for any unstable curve $W\subset \mathcal{U}_x^u$,
  $$m_W(y\in W: r^s(y)>A|W|) \geq (1-q)m_W(W).$$
  Additionally, for $x\in \mathcal{D}\backslash \Xi_{-\infty}^{\mathbb{H}}$, any $q>0$ and $A>0$ there exists a leafwise open neighborhood $\mathcal{U}^s_x\subset \mathcal{D}$ of $x$ such that for any stable curve $W\subset \mathcal{U}_x^s$,
  $$m_W(y\in W: r^u(y)>A|W|) \geq (1-q)m_W(W).$$
\end{theorem}
\section{Ergodic properties}
\label{sec:K}
The goal of this section is to prove the following.
\begin{theorem}
  \label{thm:Kproperty}
Let $\mathcal{T}$ be a repetitive, aperiodic tiling of $\mathbb{R}^2$ of finite local complexity, $\mathcal{S}$ be a $\mathcal{T}$-equivariant collection of 
dispersive scatterers with $C^3$ boundary defining a Lorentz gas of
category A, and $\mu$ an $\mathbb{R}^2$-invariant ergodic probability measure on $\Omega_\mathcal{T}$. 
Then $(f, \mathcal{D},\bar{\mu})$ has the K property.
\end{theorem}

Let us summarize the strategy of the proof in this section. Our goal will be to show that the Pinsker partition is trivial, which is one characterization of a K-system (recall that the Pinsker partition is the join of all partitions of zero entropy). 
In order to do this, we will first show that $\mu$-almost every connected component is fully contained (up to a subset of Lebesgue measure zero) in a single element of the Pinsker partition. Using the (minimal) foliated structure of the tiling space and the dynamics of $f$, it will follow that a dense union of connected components is contained in an element of the Pinsker partition. Finally, using again the structure of the tiling space and the ergodicity of $\mu$ we will show that the Pinsker partition is trivial.
We denote the Pinsker partition of $f$ by $\Pi_f$ and by $\Pi_f(x)$ the element of the Pinsker partition containing $x$.

\begin{proposition}
  \label{prop:step1}
$\mu$-almost every connected component $\mathcal{L}\subset \mathcal{D}$ is contained (mod 0) with respect to the Lebesgue measure in one element of the Pinsker partition.
\end{proposition}

Before we proceed with the proof of Proposition \ref{prop:step1}, let us
recall the partition $\xi^u$ of $\mathcal{D}$ into homogeneous unstable manifolds. This partition satisfies \cite[\S 5.1]{CM:book}:
\begin{equation}
  \label{eqn:PesinProperties}
  f^{-1} \xi^u \geq \xi^u,\hspace{.75in} \bigvee_{i\in\mathbb{Z}} f^i \xi^u = \varepsilon,\hspace{.75in}\mbox{ and consequently,}\hspace{.75in}  \bigwedge_{i \in\mathbb{Z}} f^i \xi^u \geq \Pi_f.
\end{equation}
Replacing the map $f^{-1}$ with $f$, an analogous statement holds for the stable partition $\xi^s$.

The following is an important consequence of Theorem \ref{thm:fundamental}
; the proof follows from the periodic case \cite[\S 6.5]{CM:book}.
\begin{theorem}
  \label{thm:local}
  For any point $x\in \mathcal{D}\backslash (\Xi_{\infty}^\mathbb{H}\cup \Xi_{-\infty}^\mathbb{H} )$ there is a leafwise open neighborhood $\mathcal{U}_x\subset \mathcal{D}$ with the property that Lebesgue almost every pair of points in  $\mathcal{U}_x$ can be connected by a sequence of
homogenous stable and unstable manifolds contained in $\mathcal{U}_x$.  Moreover, if $x\in\mathcal{D}$ lies on exactly one smooth singularity curve $S\subset \Xi_{\infty}^\mathbb{H}\cup \Xi_{-\infty}^\mathbb{H}$ then there is a leafwise open neighborhood $\mathcal{U}_x\subset \mathcal{D}$ with the same property.
\end{theorem}

We also need the following observation, the proof of the which is the same as in the periodic case \cite[Lemma 6.18]{CM:book}.
  \begin{lemma}
    \label{lem:shash}
    The set of points which belong on more than one smooth curve of the singularity set $\Xi_\infty^\mathbb{H}\cup \Xi_{-\infty}^\mathbb{H}$ is leafwise countable, that is, for every connected component $\mathcal{L}\subset \mathcal{D}$, the subset $\mathcal{B}_\mathcal{L}\subset \mathcal{L}$ consisting of points which belong to more than one singularity curve is countable.
  \end{lemma}

\begin{proof}[Proof of Proposition \ref{prop:step1}]
  It follows from Lemma \ref{lem:shash} that for any $x\in \mathcal{D}$ and two points $y,z\in \mathcal{L}_x\setminus \mathcal{B}_{\mathcal{L}_x}$ there exists a path $\gamma_{y,z}\subset \mathcal{L}_x\setminus \mathcal{B}_{\mathcal{L}_x}$ from $y$ to $z$. Picking $w\in \gamma_{y,z}$ and appealing to Theorem \ref{thm:local} there is a leafwise open neighborhood $\mathcal{U}_w$ such that Leb-a.e. two points in $\mathcal{U}_w$ can be joined by a sequence of stable and unstable manifolds. 
Noting that $\gamma_{y,z}$ is a compact set, we can pick finitely many points 
$w_1,\dots, w_N$ such that the union of their corresponding neighborhoods $\mathcal{U}_{w_1},\dots, \mathcal{U}_{w_N}$ covers $\gamma_{y,z}$.
  By the last property in (\ref{eqn:PesinProperties})
and the local product structure of $\bar{\mu}$,
 it follows that for $\bar{\mu}$-almost every $x$ and  Leb-a.e. $\omega\in \mathcal{L}_x$,
  \begin{equation}
    \label{eqn:statement1}
    \xi^u(\omega)\subset \Pi_f(\omega) \hspace{1in}\mbox{ and }\hspace{1in} \xi^s(\omega)\subset \Pi_f(\omega).\;\; 
  \end{equation}
Therefore by (\ref{eqn:statement1}) it follows that $y,z$, $\gamma_{y,z}$ and in fact the entire neighborhood of $\gamma_{y,z}$ defined by the union of the neighborhoods $\mathcal{U}_{w_1},\dots, \mathcal{U}_{w_N}$ is contained up to Lebesgue measure zero in the same element of the Pinsker partition $\Pi_f(y) = \Pi_f(z)$. Since we can apply the same argument to almost any two points on a connected component, we deduce that the entire connected component is contained up to Lebesgue measure zero in one element of the Pinsker partition.
\end{proof}
\begin{proposition}
\label{prop:invariance}
  For almost every connected component $\mathcal{L}$, $\Pi_f(\mathcal{L}) = \Pi_f(f(\mathcal{L}))$.
\end{proposition}
We will need the following Lemma, the proof of which is contained in the proof of \cite[Lemma 6.23]{CM:book}.
\begin{lemma}
  \label{lem:3piece}
  For any connected component $\mathcal{L}\subset \mathcal{D}$ there are two other connected components $\mathcal{L}_1$, $\mathcal{L}_2\subset \mathcal{D}$ such that the intersections
  $$f(\mathcal{L})\cap \mathcal{L}_1,\hspace{.35in} f(\mathcal{L})\cap \mathcal{L}_2, \hspace{.35in} f(\mathcal{L}_1)\cap \mathcal{L}_{2},\hspace{.35in}\mbox{ and }\hspace{.35in}f(\mathcal{L}_1)\cap \mathcal{L}$$
 are all non-empty and have positive Lebesgue measure.
\end{lemma}
\begin{proof}[Proof of Proposition \ref{prop:invariance}]
We first recall that the Pinsker partition satisfies $\Pi_f\geq \xi$ for any measurable partition $\xi$ with zero entropy. Using this property and the fact that $f(\Pi_f), f^{-1}(\Pi_f)$ have zero entropy it can be deduced that the Pinsker partition satisfies the invariance condition $\Pi_f(f(x)) = f(\Pi_f(x))$ for almost every $x$. 
Denote by $\Pi_f(\mathcal{L})$ the element of the Pinsker partition which, by Proposition \ref{prop:step1}, contains the connected component $\mathcal{L}$.
We then have that $f(\Pi_f(\mathcal{L})) = \Pi_f(f(\mathcal{L}))$. 

It will be usefull to remember that a connected component $\mathcal{L}\subset \mathcal{D}$ corresponds to a specific scatterer $\Upsilon$ in a specific tiling $\mathcal{T}\in\Omega$ along with all of its possible incoming/outgoing collisions. 
Let $\mathcal{L}_1, \mathcal{L}_2$ be as in Lemma \ref{lem:3piece}.
 From the point of view of the Lorentz gas this corresponds to scatterers $\Upsilon_1, \Upsilon_2$ on $\mathcal{T}$ along with all possible incoming/outgoing collisions such that $\Upsilon$ sends an open set of collisions to $\Upsilon_1$ and $\Upsilon_2$, and in turn
$\Upsilon_1$ sends an open set of collisions to $\Upsilon$ and $\Upsilon_2$.
  
The first part implies that there are subsets $A_i\subset \mathcal{L}_i$ of positive Lebesgue measure, $i=1,2$, such that $A_1 \cup A_2 \subset f(\Pi_f(\mathcal{L})) =\Pi_f(f(\mathcal{L}))$. However, Proposition \ref{prop:step1} forces the entire connected component to be included, and so $ \mathcal{L}_1\cup \mathcal{L}_2 \subset f(\Pi_f(\mathcal{L})) =\Pi_f(f(\mathcal{L}))$. In other words, $ \mathcal{L}_1 \cup \mathcal{L}_2$ is contained in one element of the Pinsker partition.


From the second part it follows that $\mathcal{L}\subset f(\Pi_f(\mathcal{L}_1)) = \Pi_f(f(\mathcal{L}_1))$ and $\mathcal{L}_2\subset f(\Pi_f(\mathcal{L}_1)) = \Pi_f(f(\mathcal{L}_1))$, so that 
 $ \mathcal{L} \cup \mathcal{L}_2$ is contained in one element of the Pinsker partition.
Thus we have that $\Pi_f(f(\mathcal{L})) = \Pi_f(\mathcal{L}_1) = \Pi_f(\mathcal{L}_2)  = \Pi_f(\mathcal{L})$.
\end{proof}

For $x\in \mathcal{D}$, let $\mathbb{L}_x\subset \mathcal{D}$ be the union of connected components intersecting the $\mathbb{R}^2$-orbit of $\pi(x)\in\Omega$. In other words, this is the union of all scatterers for the Lorenz gas on $\mathcal{T} = \pi(x)$, along with all possible incoming and outgoing collisions. Proposition \ref{prop:step1} and Proposition \ref{prop:invariance} now have the following consequence.

\begin{corollary}
  \label{cor:density}
  For $\bar{\mu}$-almost every $x\in\mathcal{D}$ the element of the Pinsker partition containing $x$ contains all the scatterers for the aperiodic Lorenz gas defined on $\pi(x)\in\Omega$. That is,
  $$\Pi_f(\mathbb{L}_x)\subset\Pi_f(x).$$
\end{corollary}

\begin{proof}[Proof of Theorem \ref{thm:Kproperty}]
  Corollary \ref{cor:density} implies that any element of the Pinsker partition is the union of Poincar\'e sections of Lorenz gases of different tilings in $\Omega$. In other words, for $\bar{\mu}$-almost every $x\in \mathcal{D}$ there is a subset $\mathcal{K}_x\subset \bigcup_i \mathcal{C}_i$ such that:

\begin{enumerate}
\item $\mathcal{K}_x$ is the intersection of the canonical transversal  $\mho$ with an $\mathbb{R}^2$-invariant set: $$  \mathcal{K}_x: =  \mho ~  \cap ~   \bigcup_{p\in      \Pi_f(x) }   \bigcup_{v\in\mathbb{R}^2}   \varphi_v  (  \pi(p)  ), $$

\item $    \Pi_f(x) = \bigcup_{p\in\mathcal{K}_x} \mathbb{L}_p ,$ and \\

\item $    \bar{\mu}(\Pi_f(x)) = 2\sum_i \nu(\mathcal{K}_x\cap \mathcal{C}_i)\cdot \mathrm{Leb}(\partial \mathcal{S}^i),    $

\end{enumerate}

  where we have used the local product structure of $\bar{\mu}$.   We now recall that the measure $\bar{\mu}$ is obtained from an $\mathbb{R}^2$-invariant ergodic probability measure $\mu$ on $\Omega$ with local product structure $\mathrm{Leb}\times\nu $. 
Ergodicity of $\mu$ implies that the set $\mathcal{K}_x$ either has full or null $\nu$-measure. 
Consequently, $\bar{\mu}(\Pi_f(x))\in\{0,\bar{\mu}(\mathcal{D})\}$, which implies that the Pinsker partition is trivial.
\end{proof}

\section{Planar mixing for aperiodic Lorentz gases}
\label{sec:mixing}

Throughout this section we assume that $\mu$ is an $\mathbb{R}^2$-invariant ergodic probability measure on a tiling space $\Omega$ 
of an aperiodic, repetitive tiling of $\mathbb{R}^2$ of finite local complexity
and that $\mathcal{T}\in\Omega$ is a generic tiling with respect to this measure. In particular all results in this section apply to any 
 aperiodic, repetitive tiling of $\mathbb{R}^2$ of finite local complexity
and uniform patch frequency as such tilings result in uniquely ergodic  $\mathbb{R}^2$ action on $\Omega_{\mathcal{T}}$. 

Recall that $\mu = \mbox{Vol}\times\nu$ and that by Proposition \ref{prop:invmeasure}, the corresponding $f$-invariant measure on the collision space $\mathcal{D}$ is given by $\bar{\mu}= \cos\theta \, d\rho \, d\nu \, d\theta$.

\begin{definition}
  If $\mathcal{S}\subset\mathbb{R}^2$ is a collection of $\mathcal{T}$-equivariant scatterers, then a function $g':\partial\mathcal{S}\times X\rightarrow \mathbb{R}$ is \textbf{$\mathcal{T}$-equivariant} if there exists a $\mathcal{T}$-equivariant function $g:\mathbb{R}^2\times X\rightarrow \mathbb{R}$ such that $g|_{\partial\mathcal{S}} = g'$, that is, $g'$ is the restriction to $\partial\mathcal{S}$ of $g$. A $\mathcal{T}$-equivariant function $g:\partial \mathcal{S}\times (-\pi/2,\pi/2)\rightarrow \mathbb{R}$ is \textbf{locally integrable} if for any connected component $\mathcal{L}\subset \partial \mathcal{S}$ we have that $\int_{-\pi/2}^{\pi/2} \int_\mathcal{L} g\, d\rho\cos\theta\, d\theta$ exists.

\end{definition}
We will denote the set of locally integrable, $\mathcal{T}$-equivariant functions on $\partial\mathcal{S}\times (-\pi/2,\pi/2)$ by $\Delta_\mathcal{T}(\mathcal{S})$.

\begin{lemma}
  \label{lem:bump}
For $\mu$- almost every $\mathcal{T}\in\Omega$ and any $g\in \Delta_\mathcal{T}(\mathcal{S})$ and $z\in \mathbb{R}^2$
  \begin{equation}
    \label{eqn:beta}
    \beta(g):= \lim_{T\rightarrow \infty}\frac{1}{\mbox{Vol}(B_T)}\int_{-\pi/2}^{\pi/2}\int_{B_T(z)\cap \partial\mathcal{S}} g\, d\rho\, \cos\theta\, d\theta
  \end{equation}
  exists and is independent of $z\in\partial\mathcal{S}$. In addition, there exists a transversally locally constant function  $h:\partial\Sigma \times (-\pi/2,\pi/2)\rightarrow \mathbb{R}$
such that   $g(v,\theta) = h(\varphi_{v}(\mathcal{T}),\theta)$ for all  $\theta\in (-\pi/2, \pi/2)$ and $v\in \mathbb{R}^2$
such that $\varphi_{v}(\mathcal{T})\in\partial\mathcal{S}$, and 
\begin{equation*}
  \beta(g) =  \int_{-\pi/2}^{\pi/2}    \int_{\partial \Sigma}  h  \, d\rho \, d\nu \,  \cos\theta\, d\theta =  \bar{\mu}(h).
\end{equation*}
\end{lemma}
The quantity $\beta(g)$ in (\ref{eqn:beta}) is called the planar average of $g$.
\begin{proof}
  Let $\Gamma\subset \mathbb{R}^2$ be a $C^3$ closed convex curve whose barycenter coincides with the origin. Then there exists an $\varepsilon_\Gamma>0$ such that for any $\varepsilon\in(0,\varepsilon_\Gamma)$ the $\varepsilon$-neighborhood of $\Gamma$, the open annulus $\mathcal{A}_\Gamma^\varepsilon$, can be given coordinates $(\rho, y)$ for $\rho\in S^1$ and $y\in (-\varepsilon,\varepsilon)$. As such, there exists a function $J:\mathcal{A}_\Gamma^\varepsilon\rightarrow \mathbb{R}$, positive in the interior, depending on the curvature of $\Gamma$, such that for any measurable function $f$ with support in $\mathcal{A}_{\Gamma}^\varepsilon$ we have
  $$\int_{\mathcal{A}_\Gamma^\varepsilon}f\, d\mbox{Vol} = \int_{\Gamma}\int_{-\varepsilon}^\varepsilon f(\rho,y)J(\rho,y)\, d y\, d\rho.$$
  Let $\omega:(-\varepsilon,\varepsilon)\rightarrow \mathbb{R}$ be a non-negative smooth bump function with $\omega(0)=1$ and integral 1, and consider the function $u:S^1\times (-\varepsilon,\varepsilon)\rightarrow \mathbb{R}$ defined as $u(\rho,y) = \frac{\omega(y)}{J(\rho,y)}$. Now for any $g:\Gamma\rightarrow \mathbb{R}$, let $g_\omega:\mathcal{A}_\Gamma^\varepsilon\rightarrow \mathbb{R} $ be defined as $g_\omega(\rho,y):= g(\rho)u(\rho,y)$. As such,
  \begin{equation}
    \label{eqn:bump}
    \begin{split}
      \int_{\mathcal{A}_\Gamma^\varepsilon} g_\omega\, d\mbox{Vol} &= \int_\Gamma \int_{-\varepsilon}^\varepsilon g_\omega(\rho,y) J(\rho,y)\, dy\, d\rho = \int_\Gamma \int_{-\varepsilon}^\varepsilon g(\rho) u(\rho,y) J(\rho,y)\, dy\, d\rho\\
      &= \int_\Gamma \int_{-\varepsilon}^\varepsilon g(\rho) \omega(y)\, dy\, d\rho = \int_\Gamma g(\rho)\, d\rho.
    \end{split}
  \end{equation}

  Given a $\mathcal{T}$-equivariant collection of $C^3$ scatterers $\mathcal{S}$, there is a finite set of geometrically different closed convex curves such that any connected component $\Gamma\subset \partial\mathcal{S}$ is translation-equivalent to one of these curves. Thus we can find an $\varepsilon_\mathcal{S}>0$ such that for any $\varepsilon\in(0,\varepsilon_\mathcal{S})$, the $\varepsilon$-neighborhood of any connected component of $\partial\mathcal{S}$ admits coordinates in $S^1\times (-\varepsilon,\varepsilon)$, and thus the $\varepsilon$-neighborhood of $\partial\mathcal{S}$ admits coordinates in $\mathcal{S}\times (-\varepsilon,\varepsilon)$.  Again by finite local complexity, there is a finite set of functions $\mathcal{J}_\varepsilon(\mathcal{S})=\{J_i\}$, one for each translation class of scatterers in $\mathcal{S}$, and a bump function $\omega_\mathcal{S}:(-\varepsilon,\varepsilon)\rightarrow \mathbb{R}$ so that using the $J_i$ and $\omega_\mathcal{S}$ as above we can extend a function $g\in\Delta_\mathcal{T}(\mathcal{S})$ to a $\mathcal{T}$-equivariant function $g_{\omega_\mathcal{S}}$ defined on the $\varepsilon$-neighborhood of $\partial\mathcal{S}$ in a way that preserves the integral as in (\ref{eqn:bump}). More specifically, we have that
  \begin{equation}
    \label{eqn:transfer}
    \int_{-\pi/2}^{\pi/2} \int_{B_T(z)\cap \partial\mathcal{S}} g\, d\rho\, \cos\theta\, d\theta =\int_{-\pi/2}^{\pi/2} \int_{B_T(z)}g_{\omega_\mathcal{S}}\, d\mbox{Vol} \, \cos\theta\, d\theta + \mathcal{O}(|\partial B_T(z)|),
  \end{equation}
  where the error term comes from $\varepsilon$-neighborhoods of connected components of $\partial\mathcal{S}$ that intersect $\partial B_T$.
  
 Since $g_{\omega_\mathcal{S}}$ is $\mathcal{T}$-equivariant on $\mathbb{R}^2\times (-\pi/2,\pi/2)$, by Corollary \ref{cor:relateintegrals}, there exists a transversally locally constant function  $h_\omega:\Omega\times (-\pi/2,\pi/2)\rightarrow \mathbb{R}$
such that   $g_{\omega_\mathcal{S}}(v,\theta) = h_\omega(\varphi_{v}(\mathcal{T}),\theta)$
for all $v\in\mathbb{R}^d$ and $\theta\in  (-\pi/2,\pi/2) $, and we have that
\begin{equation}
  \label{eqn:betaLim}
    \begin{split}
      \lim_{T \to \infty}  \frac{1}{\mbox{Vol}(B_T)}  \int_{-\pi/2}^{\pi/2} \int_{B_T(z)\cap \partial\mathcal{S}} g\, d\rho\, \cos\theta\, d\theta &=    \lim_{T \to \infty}  \frac{1}{\mbox{Vol}(B_T)}  \int_{-\pi/2}^{\pi/2} \int_{B_T(z)}g_{\omega_\mathcal{S}}\, d\mbox{Vol} \, \cos\theta\, d\theta \\
      &=  \int_{-\pi/2}^{\pi/2}   \int_{\Omega}  h_{\omega}(\mathcal{T}',\theta)  d\mu(\mathcal{T'}) \cos\theta\, d\theta
    \end{split}
\end{equation}
proving the first part of the lemma. To show the second part of the lemma we observe that by the definition of $g_{\omega_\mathcal{S}}$, we have that $h_\omega$ is supported in the $\varepsilon$-neighborhood of $\partial\Sigma\times (-\pi/2,\pi/2)$, and this set admits coordinates in $\partial\Sigma\times (-\pi/2,\pi/2)\times (-\varepsilon,\varepsilon)$. Furthermore, without loss of generality we refine the coordinates for $\partial\Sigma$ given in (\ref{eqn:coordinates}) as
  \begin{equation*}
    \label{eqn:coordinates2}
    \partial\Sigma = \bigsqcup_{i}\Gamma_i\times \mathcal{C}_i'
  \end{equation*}
  in such a way that
  \begin{enumerate}
  \item $\Gamma_i$ is translation equivalent to a connected component of $\mathcal{S}$,
  \item $h$ is transversally constant in each $\mathcal{C}_i'$, and
  \item $J_i'\in \mathcal{J}_\varepsilon(\mathcal{S})$ is the corresponding function with which to integrate around $\Gamma_i$ as in (\ref{eqn:bump}).
  \end{enumerate}
  Using this and letting $\Gamma_i^\varepsilon$ denote the $\varepsilon$-neighborhood of $\Gamma_i$, we have as in (\ref{eqn:bump}),
  \begin{equation*}
    \begin{split}
      \int_{-\pi/2}^{\pi/2} \int_\Omega h_\omega(\mathcal{T},\theta)\, d\mu(\mathcal{T})\,\cos\theta\,d\theta &=  \int_{-\pi/2}^{\pi/2} \sum_i \int_{\mathcal{C}_i'} \int_{\Gamma_i^\varepsilon} h_\omega\, d\mathrm{Leb}\, d\nu  \,\cos\theta\,d\theta \\
      &=  \int_{-\pi/2}^{\pi/2} \sum_i \int_{\mathcal{C}_i'} \int_{\Gamma_i} \int_{-\varepsilon}^{\varepsilon} h(\rho)\frac{\omega_\mathcal{S}(y)}{J_i'(\rho,y)}\, dy \, d\rho\, d\nu  \,\cos\theta\,d\theta \\
      &=  \int_{-\pi/2}^{\pi/2} \sum_i \int_{\mathcal{C}_i'} \int_{\Gamma_i} h(\rho) \, d\rho\, d\nu  \,\cos\theta\,d\theta \\
      &=  \int_{-\pi/2}^{\pi/2}  \int_{\partial\Sigma}h \, d\rho\, d\nu  \,\cos\theta\,d\theta =\bar{\mu}(h)
    \end{split}
  \end{equation*}
  which in conjunction with (\ref{eqn:betaLim}) gives the result.
\end{proof}

For $\theta\in (-\pi/2,\pi/2)$ and $n\in\mathbb{N}$
we denote by
    $F^n_\theta:\partial\mathcal{S}\rightarrow \partial\mathcal{S}$ the map given by the first $n$ colisions of the planar Lorentz gas with initial angle $\theta\in (-\pi/2,\pi/2)$. 
The map $F^n_\theta$ is not defined on all of $\partial\mathcal{S}\times (-\pi/2,\pi/2)$ due to discontinuities, but this is a set of Lebesgue measure zero for all $n\in\mathbb{N}$. 

\begin{theorem}
  \label{thm:generalMixing}
  Let $\Omega$ be the tiling space of an aperiodic, repetitive tiling $\mathcal{T}'$ of $\mathbb{R}^2$ of finite local complexity and let $\mu$ 
be an invariant ergodic probability measure for the $\mathbb{R}^2$ action. Let $\partial\Sigma\subset \Omega$ be a transversally locally constant set which defines for every $\mathcal{T}\in\Omega$ a 
$\mathcal{T}$-equivariant
collection $\mathcal{S}_\mathcal{T}\subset \mathbb{R}^2$ of dispersive scatterers with $C^3$-boundary of category A.

  For $\mu$-almost every $\mathcal{T}\in\Omega$, for $g_1,g_2\in \Delta_\mathcal{T}(\mathcal{S}_\mathcal{T})$ there exist transversally locally constant functions $h_1,h_2:\mathcal{D}\rightarrow \mathbb{R}$ such that for any $n\in\mathbb{N}$ and all $T$ large enough (depending on $n$), we have that for any $x\in\mathbb{R}^2$

    \begin{equation}
          \label{eqn:ApproxleafMixing}
          \int_{-\pi/2}^{\pi/2} \int_{B_T(x)\cap \partial\mathcal{S}_\mathcal{T}} g_1\circ F^n_\theta\cdot g_2 \, d\rho\, \cos\theta\, d\theta = \mathrm{Vol}(B_T)\langle h_1\circ f^n,h_2\rangle_{L^2_{\bar{\mu}}(\mathcal{D})} +  o_n(T^2),
        \end{equation}
        from which we obtain the convergence
        \begin{equation}
          \label{eqn:leafMixing}
          \lim_{n\rightarrow \infty} \lim_{T\rightarrow \infty}\frac{1}{\mathrm{Vol}(B_T)}\int_{-\pi/2}^{\pi/2} \int_{B_T(x)\cap \partial\mathcal{S}_\mathcal{T}} g_1\circ F_\theta^n\cdot g_2\, d\rho\, \cos\theta\, d\theta = \beta(g_1)\beta(g_2).
        \end{equation}


\end{theorem}
In the case of substitution tilings, we have the following.
\begin{theorem}
  \label{thm:SSmixing}
  Let $\Omega$ be the tiling space of a self-similar tiling $\mathcal{T}'$ of $\mathbb{R}^2$ of finite local complexity and let $\mu$ be the unique measure invariant under the $\mathbb{R}^2$ action. Let $\partial\Sigma\subset \Omega$ be a transversally locally constant set which defines for every $\mathcal{T}\in\Omega$ a $\mathcal{T}$-equivariant
collection $\mathcal{S}_\mathcal{T}\subset \mathbb{R}^2$ of dispersive scatterers with $C^3$-boundary of category A.

  For any $n$ there exist $d^+\geq 1$ functions $C_i^n:\Delta_\mathcal{T}(\mathcal{S})\times \Delta_\mathcal{T}(\mathcal{S})\rightarrow \mathbb{R}$ (where $d^+$ depends on the spectrum of the substitution matrix) such that for any $\mathcal{T}\in\Omega$, for $g_1,g_2\in \Delta_\mathcal{T}(\mathcal{S}_\mathcal{T})$ there exist transversally locally constant functions $h_1,h_2:\mathcal{D}\rightarrow \mathbb{R}$ such that for any $n\in\mathbb{N}$ and all $T$ large enough (depending on $n$), we have that
  \begin{equation}
    \label{eqn:SSapproxMixing}
    \begin{split}
      &\int_{-\pi/2}^{\pi/2} \int_{B_T\cap \partial\mathcal{S}_\mathcal{T}} g_1\circ F^n_\theta\cdot g_2 \, d\rho\, \cos\theta\, d\theta\\
      &\hspace{.5in}= \sum_{i=1}^{d^+}L_i(T)T^{2\frac{\log |\lambda_i|}{\log\lambda_1}} C_i^n(g_1,g_2) +  \mathcal{O}_n(T),\\
      &\hspace{.5in}= \mathrm{Vol}(B_T)\langle h_1\circ f^n,h_2\rangle_{L^2_{\bar{\mu}}(\mathcal{D})} + \sum_{i=2}^{d^+} L_i(T)T^{2\frac{\log |\lambda_i|}{\log\lambda_1}} C_i^n(g_1,g_2) +  \mathcal{O}_n(T),
    \end{split}
  \end{equation}
  where $L_i(T) = \log(T)^{s-1}$ is the size of the Jordan block associated to the $i^{th}$ eigenvalue of the substitution matrix is $s$. Moreover, the planar averaged mixing (\ref{eqn:leafMixing}) also holds.
\end{theorem}
\begin{proof}[Proof of Theorem \ref{thm:generalMixing}]
Let $\mathcal{T}\in\Omega$ be a $\mu$-generic tiling
and $\mathcal{S}$ a 
$\mathcal{T}$-equivariant
collection of dispersive scatterers with $C^3$-boundary of category A with $\bar{0}\in \partial \mathcal{S}$. If $g_1,g_2\in \Delta_{\mathcal{T}}(\mathcal{S})$, then 
by Lemma \ref{lem:bump},
there exist transversally locally constant functions $h_1,h_2:\mathcal{D}\rightarrow \mathbb{R}$ such that: $g_i(v) = h_i\circ \varphi_{v}(\mathcal{T})$ for all 
$\theta\in (-\pi/2, \pi/2)$ and
$v\in\mathbb{R}^2$ such that $\varphi_{v}(\mathcal{T})\in\partial\mathcal{S}$, 
and $\beta(g_i)= \bar{\mu}(h_i)$. 
For any $n\in \mathbb{N}$, the function $g_1\circ F_\theta^n\cdot g_2$ is a locally integrable $\mathcal{T}$-equivariant function. In fact, it is the product of two locally integrable $\mathcal{T}$-equivariant functions $g_1\circ F_\theta^n$ and $g_2$, and, by construction, the associated transversally locally functions on $\mathcal{D}$ are $h_1\circ f^n$ and $h_2$.  



By the second part of Lemma \ref{lem:bump}, we have that
 \begin{equation}\label{eqn:betan}
   \begin{split}
      \beta(g_1\circ F_\theta^n\cdot g_2) & = \int_{-\pi/2}^{\pi/2}\int_{\partial \Sigma} h_1\circ f^n \cdot h_2 \, d\rho \, d\nu \, \cos\theta\, d\theta \\
      &= \langle h_1\circ f^n,h_2\rangle_{L^2_{\bar{\mu}}(\mathcal{D})},
    \end{split}
  \end{equation}



  Thus, for large $T$ we have that
  \begin{equation}
    \label{eqn:thetaAvg}
    \begin{split}
      &\int_{-\pi/2}^{\pi/2}\int_{B_T(x)\cap\partial\mathcal{S}}g_1\circ F_\theta^n\cdot g_2\, d\rho\, \cos\theta\, d\theta \\
      &\hspace{1.5in}= \mathrm{Vol}(B_T) \int_{-\pi/2}^{\pi/2}\int_{\partial \Sigma} h_1\circ f^n \cdot h_2 \, d\rho \, d\nu\cos\theta\, d\theta + o_n(T^2) \\
      &\hspace{1.5in}= \mathrm{Vol}(B_T) \langle h_1\circ f^n , h_2\rangle_{L^2_{\bar{\mu}}(\mathcal{D})} + o_n(T^2), 
    \end{split}
  \end{equation}
  proving (\ref{eqn:ApproxleafMixing}). The leafwise mixing result (\ref{eqn:leafMixing}) follows from (\ref{eqn:betan}) by applying Theorem \ref{thm:Kproperty}. Indeed we have


 $$   \lim_{n\to \infty}  \beta(g_1\circ F_\theta^n\cdot g_2) = \lim_{n\to\infty} \langle h_1\circ f^n,h_2\rangle_{L^2_{\bar{\mu}}(\mathcal{D})} =  \bar{\mu}(h_1) \bar{\mu}(h_2) = \beta(g_1)\beta(g_2).$$

\end{proof}
\begin{proof}[Proof of Theorem \ref{thm:SSmixing}]
  To go from (\ref{eqn:ApproxleafMixing}) to (\ref{eqn:SSapproxMixing}) we need to establish rates of convergence for ergodic integrals of locally integrable, $\mathcal{T}$-equivariant functions. This essentially is done in \cite{bufetov-solomyak:tilings}, giving error rates for ergodic integrals for locally integrable, $\mathcal{T}$-equivariant functions, giving (\ref{eqn:SSapproxMixing}) the same way that (\ref{eqn:thetaAvg}) was obtained.
\end{proof}
\section{Ergodicity of the flow}
\label{sec:ergFlow}
Theorem \ref{thm:Kproperty}, which in particular implies the ergodicity of the billiard map, also implies the ergodicity of the billiard flow. In this section we prove 
Theorem \ref{thm:erg}.

 Let $\mathcal{T}$ be a repetitive, aperiodic tiling of $\mathbb{R}^2$ of finite local complexity, $\mathcal{S}$ be a $\mathcal{T}$-equivariant collection of dispersive scatterers with $C^3$ boundary of category A, and $\mu$ an $\mathbb{R}^2$-invariant ergodic probability measure on $\Omega_\mathcal{T}$.    
Let $h:\Omega\times (-\pi/2,\pi/2)\rightarrow \mathbb{R}$ be a $L^1$ transversally locally constant function which defines for every $\mathcal{T}'\in\Omega$ a $g_{\mathcal{T}'}\in\mathfrak{G}_{\mathcal{T}'}$ by $g_{\mathcal{T}'}(z,\theta) =   h  (\varphi_{z}(\mathcal{T}'),\theta)$, where $\mathfrak{G}_\mathcal{T}$ is defined in (\ref{eqn:segmentNorm}).
By the relation between $\bar{\phi}_t(\mathcal{T}',\theta)$ and $\phi_t^{\mathcal{T'}}(\bar{0},\theta)$ described in Section \ref{subsec:model}, we have for all $t\in \mathbb{R}$ that 
\begin{equation}\label{eqn:relateflows}
h\circ \bar{\phi}_t(\mathcal{T}',\theta) = g_{\mathcal{T}'} \circ \phi_t^{\mathcal{T'}}(\bar{0},\theta).
\end{equation}

Consequently, defining a transversally locally constant function $\bar{H}: \mathcal{D} \to \mathbb{R}$ by
$$ \bar{H}(\mathcal{T}',\theta) :=   \int_0^{\tau(\mathcal{T}',\theta)}  h\circ \bar{\phi}_t(\mathcal{T}',\theta)   dt ,  $$
we obtain, recalling (\ref{eqn:segmentAvg}), that
$$   \bar{H}(\mathcal{T}',\theta) =   \int_0^{\tau(\mathcal{T}',\theta)}    g_{\mathcal{T}'} \circ \phi_t^{\mathcal{T'}}(\bar{0},\theta) dt = \bar{G}_{\mathcal{T}'}(0,\theta)  . $$

Observe that for any $z\in \partial \mathcal{S}'$ we can write
$$   \bar{G}_{\mathcal{T}'}(z,\theta) =    \bar{G}_{\varphi_{z}(\mathcal{T}')}(0,\theta) = \bar{H}(  \varphi_{z}(\mathcal{T}'),\theta),   $$
so that $\beta( \bar{G}_{\mathcal{T}'})=\bar{\mu}(\bar{H}).   $ Using (\ref{eqn:relateflows}) we obtain that for any $T>0$, up to an error of size $M\|g_{\mathcal{T}'}\|_{(1)}$,
\begin{equation*}
\begin{split} 
\int_{0}^{T}  g_{\mathcal{T}'} \circ \phi_t^{\mathcal{T'}}(\bar{0},\theta) dt &= \int_{0}^{T} h\circ \bar{\phi}_t(\mathcal{T}',\theta) dt =  \sum_{i=0}^{k(T)-1} \int_{0}^{\tau(f^i(\mathcal{T}',\theta))} h\circ \bar{\phi}_t(f^i(\mathcal{T}',\theta)) dt \\
&= \sum_{i=0}^{k(T)-1}\bar{H}(f^i(\mathcal{T}', \theta))
\end{split}
\end{equation*}
where $k(T)$ denotes the number of collisions of $(\mathcal{T}',\theta)$ with the boundary $\partial\Sigma$ in time $T$. 
 Finally, by the Birkhoff ergodic theorem applied to $\bar{H}$ and $\tau$, we obtain for $\bar{\mu}$-almost every $(\mathcal{T}',\theta)\in\partial\Sigma \times [-\pi/2,  \pi/2],$
 \begin{equation*}
   \begin{split}
     &\lim_{T\to \infty}  \frac{1}{T}\int_{0}^{T}  g_{\mathcal{T}'} \circ \phi_t^{\mathcal{T'}}(\bar{0},\theta) dt = \lim_{T\to \infty} \frac{1}{T}\int_{0}^{T} h\circ \bar{\phi}_t(\mathcal{T}',\theta) dt \\
     &\hspace{1in}=\lim_{T\to \infty} \frac{k(T)}{T} \frac{1}{k(T)}\sum_{i=0}^{k(T)-1}\bar{H}(f^i(\mathcal{T}', \theta))= \bar{\tau}^{-1} \bar{\mu}(\bar{H}) =\bar{\tau}^{-1} \beta( \bar{G}_{\mathcal{T}'}).
   \end{split}
 \end{equation*}

    \bibliographystyle{amsalpha}
\bibliography{biblio}
\end{document}